\documentclass[12pt,draft,a4paper]{amsart}
\usepackage[hyphens,spaces,obeyspaces]{url}
\usepackage{units}
\usepackage[draft=false]{hyperref}
\usepackage{amsfonts}
\usepackage{soul}
\usepackage{amssymb}
\usepackage{theoremref}
\usepackage{amsmath}
\usepackage{amsthm}
\usepackage{makecell}
\usepackage{lipsum}
\usepackage{bm}
\usepackage{tikz-cd}
\usepackage{dsfont}
\usepackage{mathtools}
\usepackage{tabularx}
\usepackage{graphicx}
\usepackage{enumerate}
\usepackage{comment}
\usepackage{amscd}
\usepackage[latin2]{inputenc}
\usepackage{t1enc}
\usepackage[mathscr]{eucal}
\usepackage{indentfirst}
\usepackage{graphicx}
\usepackage{graphics}
\usepackage{pict2e}
\usepackage{epic}
\numberwithin{equation}{section}
\usepackage[margin=2.9cm]{geometry}
\usepackage{epstopdf}

\allowdisplaybreaks[4]

\theoremstyle{plain}
\newtheorem{Th}{Theorem}[section]
\newtheorem{Lemma}[Th]{Lemma}
\newtheorem{Cor}[Th]{Corollary}
\newtheorem{Prop}[Th]{Proposition}

 \theoremstyle{definition}
\newtheorem{Def}[Th]{Definition}
\newtheorem{Conv}[Th]{Convention}

\newtheorem{Rem}[Th]{Remark}
\newtheorem{Rems}[Th]{Remarks}
\newtheorem{Ex}[Th]{Examples}

\DeclareMathOperator{\Exp}{Exp}

\DeclareMathOperator{\Log}{Log}
\DeclareMathOperator{\Mot}{Mot}
\DeclareMathOperator{\rk}{rk}
\DeclareMathOperator{\tr}{tr}
\DeclareMathOperator{\res}{res}
\DeclareMathOperator{\Spec}{Spec}

\DeclareMathOperator{\Aut}{Aut}
\DeclareMathOperator{\End}{End}
\DeclareMathOperator{\GL}{GL}

\DeclareMathOperator{\BGL}{BGL}

\DeclareMathOperator{\Id}{Id}
\DeclareMathOperator{\Opf}{Opf}
\DeclareMathOperator*\colim{colim}

\newcommand{\Addresses}{{
  \bigskip
  \footnotesize

  Ruoxi Li, \textsc{Department of Mathematics, University of Pittsburgh, Pittsburgh, PA, USA}\par\nopagebreak
  \textit{E-mail address}: \texttt{rul44@pitt.edu}
}}

\title[Motivic classes of stacks in finite characteristic]{Motivic classes of stacks in finite characteristic and applications to stacks of Higgs bundles}

\author{Ruoxi Li}

\begin{document}

\begin{abstract} 
We define a ring of motivic classes of stacks suitable for symmetric powers in finite characteristic. Let $X$ be a smooth projective curve over a field of arbitrary characteristic. We calculate the motivic classes of the moduli stacks of semistable Higgs bundles on $X$. This recovers results of Fedorov, A. Soibelman and Y. Soibelman in characteristic zero, as well as those of Mozgovoy and Schiffmann for finite fields. We also obtain a simpler formula for the motivic classes of the stacks of Higgs bundles in the universal $\lambda$-ring quotient using Mellit's results.
\end{abstract}

\maketitle
\maketitle
\tableofcontents

\section{Introduction and main results}
In the seminal paper \cite{Sch}, O. Schiffmann has given explicit formulas for the number of absolutely indecomposable vector bundles on a smooth projective curve~$X$ defined over a finite field. This is essentially equivalent to calculating volumes of the stacks of Higgs bundles on $X$ (see \cite[Theorem 1.2]{Sch}, \cite[Theorem 4.9, Corollary 4.10]{MS1}). The corresponding formulas for the volumes of the stacks of Higgs bundles are given in \cite[Theorem 1.1(ii)]{MS1}.

In \cite{FSS1} similar formulas were given for \emph{motivic} volumes of the stacks of Higgs bundles. It would be nice to reprove the formulas of \cite{Sch} and \cite{MS1} by applying counting measures to motivic formulas. Unfortunately, it is assumed in \cite{FSS1} that characteristic is zero. In this paper, we fix this problem by giving the ``corrected'' definitions of rings of motivic classes of stacks so that the formulas of \cite{FSS1} become valid in finite characteristic as well. Thus, our main theorem generalizes both \cite[Theorem 1.3.3]{FSS1} and \cite[Theorem 1.1(ii)]{MS1}. We
also obtain simpler formulas for the motivic classes of the stacks of Higgs bundles in the universal $\lambda$-ring quotient using Mellit's results, see \cite{Mel1,Mel2}.
\subsection{Motivic classes of stacks}
Compared to the characteristic zero case, the difference is that we need to add a ``surjective and universally injective relation'' for the definitions of motivic classes. 
\begin{Conv}\label{Stabilizer}
All the stacks considered in this paper will be Artin stacks locally of finite type over a field $k$ such that the stabilizers of points are affine.
\end{Conv}
Recall from \cite[Tag04XG]{SP} that a point of a stack $\mathcal{X}$ is an equivalence class of morphisms $\Spec k \to\mathcal{X}$, where $k$ is a field. We denote by $\vert\mathcal{X}\vert$ the set of points of~$\mathcal{X}$. A morphism $f:\mathcal{X}\to\mathcal{Y}$ of $k$-stacks is called universally injective if for every morphism $\mathcal{Y}'\to\mathcal{Y}$, the induced map $\vert\mathcal{Y}'\times_\mathcal{Y}\mathcal{X}\vert\to\vert\mathcal{Y}'\vert$ is injective.
\begin{Def}[Definition \ref{motstacks}]\label{motstack}
Let $k$ be a field. The abelian group $\Mot(k)$ is the group generated by isomorphism classes of $k$-stacks modulo the following relations:
\begin{enumerate}[(i)]
\item $[\mathcal{X}]=[\mathcal{Y}]+[\mathcal{X}-\mathcal{Y}]$ where $\mathcal{Y}$ is a closed substack of $\mathcal{X}$.
\item $[\mathcal{X}]=[\mathcal{Y}]$ if there is a surjective and universally injective morphism $\mathcal{X}\to\mathcal{Y}$ of stacks over~$k$.
\item $[\mathcal{X}]=[\mathcal{Y}\times_k\mathbb{A}^r_k]$ where $\mathcal{X}\to\mathcal{Y}$ is a vector bundle of rank $r$.
\end{enumerate}
The class $[\mathcal{X}]$ in $\Mot(k)$ is called the \emph{motivic class} of the stack $\mathcal{X}$. The class of~$\mathbb{A}^1_k$ in $\Mot(k)$ is denoted by $\mathbb{L}$. We can give a ring structure for $\Mot(k)$ by $[\mathcal{X}]\cdot[\mathcal{Y}]\coloneqq[\mathcal{X}\times_k\mathcal{Y}]$.
\end{Def}
For $m\ge0$, let $F^m\Mot(k)$ be the subgroup generated by the classes of stacks of dimension $\le -m$. This is a ring filtration and we define the completed ring $\overline{\Mot}(k)$ as the completion of $\Mot(k)$ with respect to this filtration.
\subsection{Motivic zeta-function and plethystic exponents}
We define the motivic zeta-function for every reduced quasi-projective scheme~$Y$:
\[\zeta_{Y}(z)\coloneqq\sum_{n\ge 0}[Sym^nY]\cdot z^n\in 1+z\cdot\Mot(k)[[z]],\]
where $Sym^nY$ is the symmetric power $Y^n/S_n$ and $S_n$ is the permutation group (with the convention that $Sym^0Y=\Spec k$), see Section \ref{section22} below for more details.

We will see in Section \ref{sectioncomplete} that $\overline{\Mot}(k)$ is topologically generated by the classes $\mathbb{L}^{-n}[Y]$ where $Y$ is a reduced quasi-projective scheme. We will define $\Exp$ to be the unique continuous homomorphism \[\Exp\colon \overline{\Mot}[[z]]^{+}\to(1+\overline{\Mot}[[z]]^{+})^{\times}\]
such that \[\Exp(\mathbb{L}^{-n}[Y]z)=\zeta_Y(\mathbb{L}^{-n}z).\]
$\Exp$ is an isomorphism, so we denote by $\Log\coloneqq\Exp^{-1}$ its inverse isomorphism.

The following theorem is a generalization of \cite[Lemma 3.8.2]{FSS1} to arbitrary characteristic.
\begin{Th}[Theorem \ref{power}]
Assume that $M$ is a scheme, $\mathcal{A}_0=\Spec k$ and $\mathcal{A}_1,\mathcal{A}_2,\ldots,\mathcal{A}_n,\ldots$ are stacks. Put $A(z)=[\mathcal{A}_0]+[\mathcal{A}_1]z+[\mathcal{A}_2]z^2+\cdots+[\mathcal{A}_n]z^n+\ldots$. Then we have \[\Exp(m\Log(A(z)))=1+\sum_{k=1}^{\infty}\left\{\sum_{k_i\colon \sum ik_i=k}\left[\left(\left(M^{\sum_i k_i}\backslash\Delta\right)\times\prod_i\mathcal{A}_i^{k_i}\right)/\prod_i{S_{k_i}}\right]\right\}\cdot z^k,\]
where $\Delta$ is the ``large diagonal'' in $M^{\sum_i k_i}$ which consists of $\left(\sum_i k_i\right)$-tuples of points of $M$ with at least two coinciding ones.
\end{Th}
This theorem is crucial for calculating motivic classes of various stacks such as stacks of Higgs bundles. In finite characteristic, this theorem is not valid without the ``surjective and universal injective relations'' we have added to the definition of the ring of motivic classes.
\subsection{Counting measure}
Fix a prime power $q$. If $\mathcal{X}$ is a stack of finite type over the finite field $\mathbb{F}_q$, we denote by $\vert\mathcal{X}(\mathbb{F}_q)\vert$ the volume of the groupoid $\mathcal{X}(\mathbb{F}_q)$.

Recall that a pre-$\lambda$-ring is a commutative ring $R$
endowed with a group homomorphism $\lambda\colon R\to(1+zR[[z]])^{\times}$ such that $\lambda\equiv 1+\Id_R\ (\text{mod }z^2R[[z]])$. Recall that a $\lambda$-ring is a pre-$\lambda$-ring $R$ such that $\lambda(1)=1+z$ and for all $x,y\in R$ and $m,n\in\mathbb{Z}_{\ge0}$, $\lambda_n(xy)$ and $\lambda_n(\lambda_m(x))$ can be expressed in terms of $\lambda_i(x)$ and $\lambda_j(y)$ using the ring operations in a standard way (see \cite[Expos\'e~V, D\'efinition 2.4, (2.4.1)]{SGA6} and \cite[Chapter I, Section 1]{Knu}). Let $\overline{\Mot}^{\lambda-ring}(k)$ be the universal $\lambda$-ring quotient of $\overline{\Mot}(k)$ defined in Section \ref{section260}.

We denote by $\overline{\Mot}_{fin-type}(\mathbb{F}_q)$ the image of $\Mot(\mathbb{F}_q)$ in $\overline{\Mot}(\mathbb{F}_q)$. We want to pass from pre-$\lambda$-rings to $\lambda$-rings (see Section \ref{section260}). We denote by $\overline{\Mot}_{fin-type}^{\lambda-ring}(\mathbb{F}_q)$ the image of the composition \[\Mot(\mathbb{F}_q)\to\overline{\Mot}(\mathbb{F}_q)\to\overline{\Mot}^{\lambda-ring}(\mathbb{F}_q).\]
We have the following ring homomorphism, called the \emph{counting measure}.
\begin{Th}[Theorem \ref{countingstack}]
\begin{enumerate}[(i)]
    \item We have a unique ring homomorphism
\[\#\colon \Mot(\mathbb{F}_q)\to\mathbb{Q}\]
such that for every finite type $\mathbb{F}_q$-stack $\mathcal{X}$, $\#([\mathcal{X}])=\vert\mathcal{X}(\mathbb{F}_q)\vert$.
\item For any $A\in\Mot(\mathbb{F}_q)$, the counting measure $\#(A)$ depends only on the image of $A$ in $\overline{\Mot}^{\lambda-ring}(\mathbb{F}_q)$.
\end{enumerate}
\end{Th}
\subsection{Moduli stacks of Higgs bundles}
Fix a smooth projective geometrically connected curve $X$ over $k$. A Higgs bundle on $X$ is a pair $(E,\Phi)$ where $E$ is a vector bundle on $X$ and $\Phi\colon E\to E\otimes\Omega_X$ is a ``twisted'' $\mathcal{O}_X$-linear morphism. The rank and the degree of the pair $(E,\Phi)$ are the rank and the degree of the vector bundle $E$.

We use $\mathcal{H}igg_{r,d}$ to denote the moduli stack of rank $r$ degree $d$ Higgs bundles on~$X$. Also denote by $\mathcal{H}iggs^{ss}_{r,d}$ the open substack of semistable Higgs bundles. The stack~$\mathcal{H}iggs^{ss}_{r,d}$ is an Artin stack of finite type. We refer the reader to
Section \ref{section41} for more details.
\subsection{Explicit formulas for motivic classes}
We have a similar result compared to \cite[Theorem 1.3.3]{FSS1} for the motivic classes
of moduli stacks of semistable Higgs bundles. In particular, for $r\ge0$ we define classes $H_{r,d}\in\overline{\Mot}(k)$ by complicated but explicit formulas, see Sections \ref{section44} and \ref{section45}. The following theorem gives an explicit formula for the motivic classes of the stacks of Higgs bundles.
\begin{Th}[Theorem \ref{higgsformula}]\label{higgsformulas}
Let $k$ be a field of arbitrary characteristic and let $X$ be a smooth projective geometrically connected curve over $k$. Assume that there exists a $k$-rational divisor of degree $1$ on $X$. For any $r>0$ and any $d$, if $e$ is large enough then we have in $\overline{\Mot}(k)$
\[[\mathcal{H}iggs^{ss}_{r,d}(X)]=H_{r,d+er}.\]
\end{Th}
In the finite field case, the above theorem implies \cite[Theorem 1.1(ii)]{MS1}, see Remark \ref{MS0}. When $r$ and $d$ are coprime, see also \cite[Theorem 1.2]{Sch}.
\subsection{Mellit's simplification}\label{section14}
Recall that $X$ is a smooth projective geometrically connected curve over $k$. We now give the definition of the motivic $L$-function, which is a polynomial of degree~$2g$, where $g\coloneqq g(X)$ is the genus of $X$, see \cite[Theorem~1.1.9]{Kap} and \cite[Proposition~1.3.1]{FSS1}:
\begin{equation}\label{Lfunction}
L^{mot}_X(z)\coloneqq\zeta_X(z)(1-z)(1-\mathbb{L}z)\in\Mot(k)[z].
\end{equation}
We will abuse notation and use $L^{mot}_X(z)$ to denote its image in $\overline{\Mot}^{\lambda-ring}(k)[z]$.

Set
\begin{equation}\label{omegamot}
\Omega_X^{mot}\coloneqq\sum_{\lambda}w^{\vert\lambda\vert}\mathbb{L}^{g\langle\lambda,\lambda\rangle}\prod_{\square\in\lambda}\frac{L^{mot}_X(z^{2a(\square)+1}\mathbb{L}^{-l(\square)-1})}{(z^{2a(\square)+2}-\mathbb{L}^{l(\square)})(z^{2a(\square)}-\mathbb{L}^{l(\square)+1})}\in\overline{\Mot}^{\lambda-ring}(k)[[z,w]]. 
\end{equation}
Here the sum is over all partitions $\lambda=(\lambda_1\ge\ldots\ge\lambda_i\ge\ldots)$ and the product is over all boxes of the partition, $a(\square)$ and $l(\square)$ are the arm and the leg of the box respectively, $\langle\lambda,\lambda\rangle\coloneqq\sum_i(\lambda'_i)^2$ where $\lambda'=(\lambda'_1\ge\ldots\ge\lambda'_i\ge\ldots)$ is the conjugate partition and $\vert\lambda\vert\coloneqq\sum_i\lambda_i$. Note that we expand the denominators in positive powers of~$z$.

In Section \ref{realsection3}, we will define inverse bijections
\[\Exp\colon \overline{\Mot}^{\lambda-ring}(k)[[z,w]]^{+}\to(1+\overline{\Mot}^{\lambda-ring}(k)[[z,w]]^{+})^{\times}\]
\[\Log\colon (1+\overline{\Mot}^{\lambda-ring}(k)[[z,w]]^{+})^{\times}\to\overline{\Mot}^{\lambda-ring}(k)[[z,w]]^{+},\]
where $\overline{\Mot}^{\lambda-ring}(k)[[z,w]]^{+}$ stands for the ideal of power series with vanishing constant term.

Set 
\[\mathbb{H}^{mot}_X\coloneqq(1-z^2)\Log\Omega^{mot}_X.\]
In characteristic zero, the following theorem is a particular case of \cite[Theorem~1.13.1]{FSS3} when $\delta=0,D=\emptyset$ and $\varepsilon=0$. Using our formulas in Theorem \ref{higgsformulas}, we will show that the theorem is valid in arbitrary characteristic.  It gives a formula for the motivic classes of
moduli stacks of semistable Higgs bundles in $\overline{\Mot}^{\lambda-ring}(k)$.
\begin{Th}[Theorem \ref{degenerate1}]\label{degenerate}
Let $d/r=d_0/r_0$ where $r_0$ and $d_0$ are coprime integers with $r_0>0$. The image of the motivic class $[\mathcal{H}iggs_{r,d}^{ss}]$ in $\overline{\Mot}^{\lambda-ring}(k)$ is equal to the coefficient at~$w^r$ in
\[
\mathbb{L}^{(g-1)r^2}\Exp\left(\mathbb{L}\left(\mathbb{H}^{mot}_X\Big\vert_{z=1}\right)_{[r_0]}\right),
\]
where $A(w)_{[r_0]}$ stands for the sum of the monomials whose exponents are multiples of~$r_0$.
\end{Th}
Using the theorem above, we can derive the counting results of Schiffmann and Mozgovoy (\cite{Sch,MS1}), see Theorem \ref{finitefieldhiggs}.
\subsection{Organization of the article}The paper is organized as follows: in Section~\ref{section2}, we define and study the motivic classes of stacks in arbitrary characteristic; in Section~\ref{count}, we define the counting measure for the ring of motivic classes and use the \'etale realization to show that the counting measure only depends on the image of the motivic class in the completed $\lambda$-ring; in Section~\ref{realsection3}, we define the power structures on the rings of motivic classes of schemes and stacks; in Section \ref{section4}, we calculate the motivic classes of the stacks of semistable Higgs bundles; finally in Section \ref{section5}, we use Mellit's results \cite{Mel1,Mel2} to obtain a simpler formula for the motivic classes of stacks of Higgs bundles in the universal $\lambda$-ring quotient and to calculate their volumes in the case when the field is finite.
\subsection{Acknowledgements}
The author would like to thank his advisor Roman Fedorov for providing plenty of help
and encouragement throughout this project. The author is thankful to Adrian Vasiu for suggesting Theorem \ref{perfectfield}. The author thanks Sergey Mozgovoy, Markus Reineke, Evgeny Shinder, Rahul Singh and Minghao Zhao for valuable discussions. The author is partially supported
by NSF grants DMS-2001516 and DMS-2402553.
\section{Motivic classes in finite characteristic}\label{section2}
\subsection{Definitions of the rings of motivic classes of schemes and stacks}\label{section21}
In this section, we give basic definitions and results about motivic classes of schemes and stacks. Our definition is new in finite characteristic. We recall that the motivic classes of varieties have been studied in \cite{BM,Got,Kap,LL1,LL,Mus}; the motivic classes of stacks have been studied in \cite{BD,BM,Eke,FSS1,Joy,Toe}; the motivic functions have been studied in \cite{FSS1,KS1}.

Recall the definition of $\Mot(k)$, see Definition \ref{motstack}.
\begin{Def}\label{motstacks}
Let $k$ be a field. The abelian group $\Mot(k)$ is the group generated by isomorphism classes of $k$-stacks modulo the following relations:
\begin{enumerate}[(i)]
\item $[\mathcal{X}]=[\mathcal{Y}]+[\mathcal{X}-\mathcal{Y}]$ where $\mathcal{Y}$ is a closed substack of $\mathcal{X}$.
\item $[\mathcal{X}]=[\mathcal{Y}]$ if there is a surjective and universally injective morphism $\mathcal{X}\to\mathcal{Y}$ of stacks over~$k$.
\item $[\mathcal{X}]=[\mathcal{Y}\times_k\mathbb{A}^r_k]$ where $\mathcal{X}\to\mathcal{Y}$ is a vector bundle of rank $r$.
\end{enumerate}
The class $[\mathcal{X}]$ in $\Mot(k)$ is called the \emph{motivic class} of the stack $\mathcal{X}$. The class of~$\mathbb{A}^1_k$ in $\Mot(k)$ is denoted by $\mathbb{L}$. We can give a ring structure for $\Mot(k)$ by $[\mathcal{X}]\cdot[\mathcal{Y}]\coloneqq[\mathcal{X}\times_k\mathcal{Y}]$.
\end{Def}
\begin{Rems}\label{remark1}
\begin{enumerate}[(i)]
\item We need the affine stabilizer condition in Convention \ref{Stabilizer} since by \cite[Lemma 3.5.1, Propositions 3.5.6 and 3.5.9]{Kre}, every such stack has a stratification by global quotients of the form $X/\GL(n)$ with $X$ a scheme. This will be used in the proof of Theorem \ref{motequal}.
\item Note that $[\emptyset]=0$. Note also that for any stack $\mathcal{X}$ we have $[\mathcal{X}_{red}]=[\mathcal{X}]$. Indeed, $\mathcal{X}-\mathcal{X}_{red}=\emptyset$ and we can use the first relation of Definition \ref{motstacks}.
\item We will explain why we need to add the second relation in Lemma \ref{G-invariant}.
\item With the second relation of Definition \ref{motstacks}, we can replace the first relation with apparently a weaker one $[\mathcal{X}]=\sum_{i=1}^n[\mathcal{X}_i]$ if $\mathcal{X}$ is the disjoint union of substacks~$\mathcal{X}_i$.
\item The second relation is non-trivial by the following example. Fix a perfect field~$k$ where $p\coloneqq\text{char}(k)\neq 0,2$. For an elliptic curve $E$ over $k$, denote by $E^{(p)}$ the Frobenius twist of $E$. The relative Frobenius $E\to E^{(p)}$ is surjective and by \cite[Tag0CCB]{SP}, it is universally injective, so we have
\[[E^{(p)}]=[E].\]
Recall the $j$-invariant of $E$ from \cite[Chapter 4, Section~4]{Har}. For any elliptic curve $E$ given in Legendre form as $y^2=x(x-1)(x-\lambda)$, $E^{(p)}$ is of Legendre form $y^2=x(x-1)(x-\lambda^p)$, thus
\[j(E^{(p)})=2^8\frac{\left(\left(\lambda^p\right)^2-\lambda^p+1\right)^3}{\left(\lambda^p\right)^2(\lambda^p-1)^2}=(j(E))^p.
\]
By \cite[Chapter 4, Theorem 4.1(b,c)]{Har}, in general $E^{(p)}\not\simeq E$ unless $k=\mathbb{Z}/p\mathbb{Z}$.
\item For the motivic classes of schemes defined in Definition \ref{motscheme} below, the vector bundle relation is trivial since every vector bundle over a scheme is locally trivial in the Zariski topology. However in the case of stacks, the universal vector bundle \[V\coloneqq\mathbb{A}^n_k\times^{\GL(n)}\{pt\}\to\BGL(n)\]
is not Zariski locally trivial.
\item When $k$ has characteristic zero, the ``surjective and universally injective relation'' is not needed: without this relation, one obtains an isomorphic ring of motivic classes. This follows by combining Remark \ref{rem24}(iii) and Theorem \ref{motequal} below with an isomorphism in \cite[Section 2.5]{FSS1}.
\end{enumerate}
\end{Rems}
The following definition for the ring of motivic classes of schemes is similar to the definition of $\widetilde{K}_0(\text{Var}/k)$ in \cite[Section 7.2]{Mus}.
\begin{Def}\label{motscheme}
For any field $k$, the abelian group $\Mot_{sch}(k)$ is the group generated by isomorphism classes of schemes of finite type over $k$ modulo the following relations:
\begin{enumerate}[(i)]
\item $[X]=[Y]+[X-Y]$ where $Y$ is a closed subscheme of $X$.
\item $[X]=[Y]$ if there is a surjective and universally injective morphism $X\to Y$ of schemes over $k$.
\end{enumerate}
The class $[X]$ in $\Mot_{sch}(k)$ is called the \emph{motivic class} of the scheme $X$. The class of~$\mathbb{A}^1_k$ in $\Mot_{sch}(k)$ is also denoted by $\mathbb{L}$ by abuse of notation. Similarly, we can give a ring structure for $\Mot_{sch}(k)$ by $[X]\cdot[Y]\coloneqq[X\times_k Y]$. 
\end{Def}
\begin{Rems}\label{rem24}
\begin{enumerate}[(i)]
\item In \cite[Section 7.2]{Mus}, Musta{\c t}{\u a} defines the ring of motivic classes of varieties (i.e., reduced schemes of finite type) $\widetilde{K}_0(\text{Var}/k)$ similarly to Definition \ref{motscheme} with the only difference that the multiplication is given by $[X]\cdot[Y]\coloneqq[(X\times_k Y)_{red}]$ (the reason for this modification is that the product of two reduced schemes is not always reduced). This ring is, however, isomorphic to $\Mot_{sch}(k)$ because $[X]=[X_{red}]$ for a scheme $X$.
\item Musta{\c t}{\u a} also defines the ring of motivic classes of quasi-projective varieties $\widetilde{K}^{qpr}_0(\text{Var}/k)$. This ring is also isomorphic to the previous one, see Lemma \ref{qpr}.
\item If $\text{char}(k)=0$, then the ring of motivic classes of schemes coincides with the Grothendieck ring of varieties (see \cite[Proposition 7.25]{Mus}). That is, the second relation of Definition \ref{motscheme} is automatically true because every surjective and universally injective morphism becomes an isomorphism after passing to locally closed stratifications (see \cite[Proposition A.24]{Mus} and \cite[Tag01S4]{SP}).
\end{enumerate}
\end{Rems}
\begin{Lemma}\label{Eke}
In $\Mot(k)$ we have the following results.
\begin{enumerate}[(i)]
\item $[\GL(n)]=(\mathbb{L}^n-1)(\mathbb{L}^n-\mathbb{L})\cdots(\mathbb{L}^n-\mathbb{L}^{n-1})$.
\item For a stack $\mathcal{X}$, we have $[\mathcal{X}/\GL(n)]=[\mathcal{X}]/[\GL(n)]$.
\item $[\BGL(n)]={1}/\left((\mathbb{L}^n-1)(\mathbb{L}^n-\mathbb{L})\cdots(\mathbb{L}^n-\mathbb{L}^{n-1})\right)$.
\end{enumerate}
\end{Lemma}
\begin{proof}
For (i) and (ii) see \cite[Proposition 1.1, i),ii)]{Eke}. Now (iii) follows from~(ii).
\end{proof}
The following theorem is similar to \cite[Th\'eor\`eme 3.10]{Toe} and \cite[Theorem~1.2]{Eke}. We give a proof for the sake of completeness.
\begin{Th}\label{motequal}
There is a natural ring isomorphism
\[\Mot_{sch}(k)[\mathbb{L}^{-1},(\mathbb{L}^{i}-1)^{-1}\vert i>0]\cong\Mot(k).\]
\end{Th}
\begin{proof}
\emph{Step 1.} We first construct a homomorphism from the left hand side to the right hand side. We have an obvious homomorphism from $\Mot_{sch}(k)$ to $\Mot(k)$. By Lemma \ref{Eke}(iii), the images of $\mathbb{L}$ and $\mathbb{L}^{i}-1$ for $i\ge 1$ are invertible in $\Mot(k)$. Now by the universal property of localization we obtain a ring homomorphism 
\[A\colon \Mot_{sch}(k)[\mathbb{L}^{-1},(\mathbb{L}^{i}-1)^{-1}\vert i>0]\to\Mot(k).\]
\emph{Step 2.} Next, we define a map $B$ assigning to a $k$-stack $\mathcal{X}$ an element  $B(\mathcal{X})\in\Mot_{sch}(k)[\mathbb{L}^{-1},(\mathbb{L}^{i}-1)^{-1}\vert i>0]$ as follows: according to Convention \ref{Stabilizer} and \cite[Lemma 3.5.1, Propositions 3.5.6 and 3.5.9]{Kre}, every stack $\mathcal{X}$  admits a stratification   $\mathcal{X}=\bigsqcup_i\left(X_i/\GL(n_i)\right)$ where $X_i$ are schemes. We use Lemma \ref{Eke}(i) and define $B$ by 
\[B([\mathcal{X}])=\sum_i\frac{[X_i]}{[\text{GL}(n_i)]}.\]
If the stack $\mathcal{X}$ has two stratifications by global quotients 
\[\mathcal{X}=\bigsqcup_iX_i/\GL(n_i)\text{ and }\mathcal{X}=\bigsqcup_jY_j/\GL(m_j),\] then
\[\sum_i[X_i]/[\GL(n_i)]=\sum_j[Y_j]/[\GL(m_j)]\in\Mot_{sch}(k)[\mathbb{L}^{-1},(\mathbb{L}^{i}-1)^{-1}\vert i>0],\]
which follows from the proof of \cite[Lemma 2.3]{BD}. Thus $B$ is well-defined.

\emph{Step 3.} We want to show that $B$ induces a group homomorphism 
\[\tilde{B}\colon \Mot(k)\to\Mot_{sch}(k)[\mathbb{L}^{-1},(\mathbb{L}^{i}-1)^{-1}\vert i>0],\]
by checking that it respects the relations of Definition \ref{motstacks}. Notice that the first relation is obviously preserved.

For the second relation, we need to show that if $\mathcal{X}\to\mathcal{Y}$ is surjective and universally injective, then $B(\mathcal{X})=B(\mathcal{Y})$. By \cite[Tag03MH and Tag03MW]{SP}, surjective and universally injective morphisms are stable under base change, so the relation reduces to the following statement: if $\mathcal{X}\to Y/\GL(n)$ is a surjective and universally injective morphism of stacks with $Y$ being an algebraic space, then
\[B(\mathcal{X})=B(Y/\text{GL}(n)).\]
We have the following Cartesian diagram where $\mathcal{X'}\to\mathcal{X}$ is a $\GL(n)$-torsor
\[
\begin{tikzcd}
\mathcal{X'}\arrow{d}\arrow{r}{\GL(n)}  & \mathcal{X}\arrow{d}\\
Y\arrow{r}&     Y/\GL(n).
\end{tikzcd}
\]
By Lemma \ref{Eke}(ii), we reduce to the case of a surjective and universally injective morphism to a scheme: \[f\colon \mathcal{X}'\to Y.\]
By \cite[Tag04XB]{SP}, a univerally injective morphism of stacks is representable by algebraic spaces. In particular, $\mathcal{X}'$ is an algebraic space.

By Noetherian induction, we can assume that $Y$ is integral. Let $\xi$ be the generic point of $Y$ and let $\eta$ be the unique generic point of $\mathcal{X}'$ mapping to $Y$. By \cite[Tag06NH]{SP}, there is an open neighborhood $Z\subset\mathcal{X}'$ of $\eta$ such that $Z$ is a scheme. Then $f(\mathcal{X}'-Z)$ is a constructible subset of $Y$ not containing $\xi$, so the closure of $f(\mathcal{X}'-Z)$ is a proper subset of $Y$. Let $U$ be the complement of this closure.

Note that since $f^{-1}(U)\subset Z$, we have 
\[f^{-1}(U)=\mathcal{X}'\times_YU=Z\times_YU,\]
which is a scheme. Also $f^{-1}(U)\to U$ is surjective and universally injective, thus $[f^{-1}(U)]=[U]$ and we are reduced to showing that 
\[[\mathcal{X}'-f^{-1}(U)]=[Y-U],\]
which follows by the induction on dimension.

For the third relation, it is enough to show that if $\mathcal{E}\to X/\GL(n)$ is a vector bundle of rank $r$, then $B(\mathcal{E})=\mathbb{L}^rB(X/\GL(n))$. We have the following Cartesian diagram
\[
\begin{tikzcd}
E\arrow{r}\arrow{d}  & X\arrow{d}\\
\mathcal{E}\arrow{r}{} &     X/\GL(n),
\end{tikzcd}
\]
where $\mathcal{E}\cong E/\GL(n)$ and $E$ is a scheme. Thus by Remark \ref{remark1}(vi)
\[B(\mathcal{E})=[E]/[\GL(n)]=\mathbb{L}^r[X]/[\GL(n)]=\mathbb{L}^rB(X/\GL(n)).\]
\emph{Step 4.} By Lemma \ref{Eke}(ii) 
\[A\circ \tilde{B}=\Id_{\Mot(k)}.\]
Now it suffices to show that $\tilde{B}$ is a surjection. It follows from Lemma \ref{Eke}(iii) that any finite product of elements of the form $\mathbb{L}^i-1$ and $\mathbb{L}$ divides $[\GL(n)]$, when $n$ is large enough. Thus, every element in the localized ring $\Mot_{sch}(k)[\mathbb{L}^{-1},(\mathbb{L}^{i}-1)^{-1}\vert i>0]$ can be written as $([X]-[Y])/[\GL(n)]$, where $X$ and $Y$ are schemes and $n$ is large enough. Since \[B((X-Y)\times_k\BGL(n))=([X]-[Y])/[\GL(n)],\]
we are done.
\end{proof}
\subsection{Dimensional completion}\label{sectioncomplete}
Following Behrend and Dhillon's results in \cite[Section~2.1]{BD}, we define the dimensional completion of $\Mot(k)$.
\begin{Def}\label{completedring}
For $m\ge0$, let $F^m\Mot(k)$ be the subgroup generated by the classes of stacks of dimension $\le -m$. This is a ring filtration and we define the completed ring $\overline{\Mot}(k)$ as the completion of $\Mot(k)$ with respect to this filtration.
\end{Def}
We note that calculations of motivic classes of various stacks associated with a curve are based on the ``motivic Harder formula'' proved in \cite[Section 4]{FSS1}, this formula is only valid in the completed ring of motivic classes.

The following definition is similar to \cite[Section 1.1]{FSS1}.
\begin{Def}\label{completed}
For $m\ge 0$, let $F^m\Mot_{sch}(k)[\mathbb{L}^{-1}]$ be the subgroup of $\Mot_{sch}(k)[\mathbb{L}^{-1}]$ generated by $[X]/\mathbb{L}^n$ with $\dim X-n\le -m$. We define the completed ring $\overline{\Mot}_{sch}(k)$ as the completion of the localization $\Mot_{sch}(k)[\mathbb{L}^{-1}]$ with respect to this filtration.
\end{Def}
\begin{Prop}\label{completioniso}
We have a natural isomorphism of topological rings
\[\overline{\Mot}(k)\xrightarrow{\sim}\overline{\Mot}_{sch}(k).\]
\end{Prop}
\begin{proof}
Note that $\mathbb{L}$ and $\mathbb{L}^i-1$ are invertible in $\overline{\Mot}_{sch}(k)$ (with the inverse being $\sum_{n>0}\mathbb{L}^{-in}$), so using Theorem \ref{motequal} we obtain a homomorphism $\Mot(k)\to\overline{\Mot}_{sch}(k)$. This homomorphism is continuous because it takes $F^m\Mot(k)$ to $F^m\Mot_{sch}(k)[\mathbb{L}^{-1}]$. Next, we extend this homomorphism by continuity to a homomorphism
\[\overline{\Mot}(k)\to\overline{\Mot}_{sch}(k).\]
Let us construct the inverse homomorphism. Since $\mathbb{L}$ is invertible in $\Mot(k)$, the ring homomorphism $\Mot_{sch}(k)\to\Mot(k)$ extends to $\Mot_{sch}(k)[\mathbb{L}^{-1}]\to\Mot(k)$.

Since this homomorphism takes $F^m\Mot_{sch}(k)[\mathbb{L}^{-1}]$ to $F^m\Mot(k)$, it extends by continuity to a continous homomorphism
\[\overline{\Mot}_{sch}(k)\to\overline{\Mot}(k).\]
It is easy to check that the two homomorphisms are inverse to each other (note that it is enought to check on dense open subsets of the topological rings).

For the motivic classes over a field of characteristic $0$ see also \cite[Section~2.5]{FSS1}.
\end{proof}
\subsection{Motivic zeta-function and pre-$\lambda$-ring structure on $\overline{\Mot}(k)$}\label{section22}
The main goal of this section is to define the pre-$\lambda$-ring structure on $\overline{\Mot}(k)$ and the main result is Theorem \ref{prestructure}. We recall the definitions of quotient by finite groups following \cite[Section A.1]{Mus} and motivic zeta-function from \cite[(1.3)]{Kap}. Then we recall the definition of pre-$\lambda$-ring structure, see \cite[Chapter I]{Knu}. In this paper we only consider pre-$\lambda$-ring structures on commutative rings with units.

First we follow \cite[Section A.1]{Mus} to recall the definition of quotients by finite groups.

Let $X$ be a scheme of finite type over a field $k$ and let $G$ be an abstract finite group, acting on $X$ by algebraic automorphisms over $k$ (we make the convention that when $X=\Spec A$, the group $G$ acts on the ring $A$ on the left as $g\cdot a$, so it acts on $\Spec A$ on the right respectively). We denote by $\sigma_g$ the automorphism of $X$ corresponding to $g\in G$. 
\begin{Def}A quotient of $X$ by $G$ is a morphism $\pi\colon X\to W$ with the following two properties:
\begin{enumerate}[(i)]
\item $\pi$ is $G$-invariant, that is $\pi\circ\sigma_g=\pi$ for every $g\in G$.
\item $\pi$ is universal with this property: for every scheme $Z$ over $k$, and every $G$-invariant morphism $f\colon X\to Z$, there is a unique morphism $h\colon W\to Z$ such that $h\circ\pi=f$.
\end{enumerate}
It follows from this universal property that if a quotient exists, then it is unique up to isomorphism. In this case, we write $W=X/G$.
\end{Def}
\begin{Lemma}\label{quotient}
If $X$ is a quasi-projective scheme, then the quotient $X/G$ exists.
\end{Lemma}
\begin{proof}
See \cite[Section A.1]{Mus}, especially the discussions of affine open covers after \cite[Corollary~A.3]{Mus}.
\end{proof}
We have the following lemma.
\begin{Lemma}\label{G-invariant}
Let $X$ be a quasi-projective scheme over $k$, and $G$ an abstract finite
group acting on $X$ by algebraic automorphisms over $k$. Let $X=\bigsqcup_i X_i$ be a stratification with $G$-invariant locally closed subschemes~$X_i$. We have in $\Mot_{sch}(k)$
\[[X/G]=\sum_i[X_i/G].
\]
\end{Lemma}
\begin{proof}
Let $X=\bigsqcup_i X_i$ be a stratification of $X$ with $G$-invariant locally closed subschemes~$X_i$. Let $\pi\colon X\to X/G$
be the quotient morphism. We have
\[X/G=\bigsqcup_i\pi(X_i).\]
For any $X_i$, there is a natural surjective and universally injective morphism $X_i/G\to\pi(X_i)$ by \cite[Proposition A.25]{Mus} and \cite[Tag01S4]{SP}. Thus we have
\[[X_i/G]=[\pi(X_i)].\]
Finally 
\[[X/G]=\sum_i[\pi(X_i)]=\sum_i[X_i/G].\]
\end{proof}
We define the motivic zeta-function for every reduced quasi-projective scheme~$X$:
\[\zeta_X(z)\coloneqq\sum_{n\ge 0}[Sym^nX]\cdot z^n\in 1+z\cdot\Mot_{sch}(k)[[z]],\]
where $Sym^nX$ is the symmetric power $X^n/S_n$ where $S_n$ is the permutation group (with the convention that $Sym^0X=\Spec k$).
\begin{Rem}\label{countzeta}
For $k=\mathbb{F}_q$ a finite field with $q$ elements, the image of the series $\zeta_X(z)$ under the counting measure in Theorem \ref{countingstack}
coincides with the local zeta-function, see \cite[Proposition 7.31]{Mus}.
\end{Rem}
\begin{Prop}\label{zeta}
If $X$ is a reduced quasi-projective scheme over $k$, $Y$ is a reduced closed subscheme of $X$ and $n\ge 0$. Then
\begin{enumerate}[(i)]
\item $\zeta_X(z)=\zeta_Y(z)\zeta_{X-Y}(z)$.
\item $\zeta_{X\times_k\mathbb{A}_k^n}(z)=\zeta_X(\mathbb{L}^nz)$.
\end{enumerate}
\end{Prop}
\begin{proof}
This material is well-known (see e.g., \cite[Propositions 7.28 and~7.32]{Mus}). To illustrate the importance of the second relation of Definition \ref{motstacks}, we sketch a proof of (i) following \cite[Proposition 7.28 and Lemma 7.29]{Mus}). If $k$ has characteristic~$0$, the proof of (ii) is also given in \cite[Statement~3]{GZLMH1}.

For fixed $n$, let $i,j$ be nonnegative integers such that $i+j=n$. Denote by $W^{i,j}$ the locally closed subset of $X^n$ given by $\cup_{g\in S_n}(Y^i\times_k(X-Y)^j)g$.

This gives a stratification of $X^n$ by reduced locally closed subschemes preserved by the $S_n$-action. There is a natural quotient morphism $\pi\colon X^n\to Sym^nX$.

By \cite[Propositions A.7 and A.8]{Mus}, we have an isomorphism
\[W^{i,j}/S_n\cong Sym^iY\times_k Sym^j(X-Y).\]
Now by Lemma \ref{G-invariant} we have
\[[Sym^n(X)]=\sum_{i+j=n}[W^{i,j}/S_n]=\sum_{i+j=n}[Sym^iY]\cdot[Sym^j(X-Y)].
\]
Finally 
\[\zeta_X(z)=\sum_{n\ge 0}[Sym^n(X)]z^n=\sum_{n\ge 0}\sum_{i+j=n}[Sym^iY]\cdot[Sym^j(X-Y)]z^{i+j}=\zeta_Y(z)\zeta_{X-Y}(z).
\]
\end{proof}
We recall the definition of pre-$\lambda$-ring structure, see \cite[Chapter I]{Knu}.
\begin{Def}\label{prelambda}
A homomorphism of abelian groups 
\[R\to\left(1+z\cdot R[[z]]\right)^{\times}\] 
\[\quad\quad\quad\quad\quad\quad\quad\quad\quad\quad\quad a\mapsto\lambda_a(z)=1+az+\text{terms of degree at least }2\]
is called a pre-$\lambda$-ring structure. In other words, for all $a,b\in R$ we have
\begin{enumerate}[(i)]
\item $\lambda_a(z)\equiv1+az\ (\text{mod}\ z^2)$.
\item $\lambda_{a+b}(z)=\lambda_a(z)\lambda_b(z)$.
\end{enumerate}
\end{Def}
Now we want to define a pre-$\lambda$-ring structure on $\overline{\Mot}(k)$ similarly to \cite[Section~1.3]{FSS1}. First, we need a lemma.
\begin{Lemma}\label{prelambdaex}
Given a pre-$\lambda$-ring structure on a ring $R$ and $l\in R\backslash\{0\}$, if for all $r\in R$ we have
\[\lambda_{l r}(z)=\lambda_r(lz),\]
then we have a unique pre-$\lambda$-ring structure on $R[l^{-1}]$ such that the homomorphism $R\to R[l^{-1}]$ is a homomorphism of pre-$\lambda$-rings, i.e, the homomorphism commutes with the pre-$\lambda$-ring structure.
\begin{proof}
\emph{Claim. Let $a\in R$ such that $l^Na=0$ for some $N>0$, then $\lambda_a(z)=1.$} Indeed
\[\lambda_a(l^Nz)=\lambda_{l^Na}(z)=\lambda_0(z)=1,\]
which implies the claim.

Consider $a\in R[l^{-1}]$, then $a$ is of the form $a'/l^i$, where $a'\in R$. Set $\lambda_a\coloneqq\lambda_{a'}(l^{-i}z)$. We need to check that $\lambda$ is well-defined. Indeed, if $a'/l^i=a''/l^j$, then
\[l^N(l^ja'-l^ia'')=0\] 
for some $N$, so by claim $\lambda_{l^ja'-l^ia''}(z)=1$, which also implies that \[\lambda_{a'/l^i}(z)=\lambda_{a''/l^j}(z).\]
It is clear that we obtained a pre-$\lambda$-ring structure on $R[l^{-1}]$ and the homomorphism $R\to R[l^{-1}]$ is a homomorphism of pre-$\lambda$-rings.
\end{proof}
\end{Lemma}
For any field $k$, we define $\Mot^{qpr}_{sch}(k)$ as the abelian group generated by isomorphism classes of reduced quasi-projective schemes of finite type over $k$ modulo relations of Definition~\ref{motscheme}. We have a group homomorphism $\Phi\colon \Mot^{qpr}_{sch}(k)\to\Mot_{sch}(k)$ such that $\Phi([X])=[X]$ for any reduced quasi-projective scheme $X$. In fact, this is an isomorphism.
\begin{Lemma}\label{qpr}
The group homomorphism $\Phi$ is an isomorphism.
\end{Lemma}
\begin{proof}
See \cite[Proposition 7.27]{Mus}.
\end{proof}
\begin{Th}\label{prestructure}
There is a unique continuous pre-$\lambda$-ring structure on $\overline{\Mot}(k)$ that sends the class of a reduced quasi-projective scheme $X$ to $(\zeta_X(-z))^{-1}$ and for any $A\in\overline{\Mot}(k)$ and any $n\in\mathbb{Z}$ we have $\zeta_{\mathbb{L}^nA}(z)=\zeta_A(\mathbb{L}^nz)$.
\end{Th}
\begin{proof}
Similarly to \cite[Proposition 7.28]{Mus}, a pre-$\lambda$-ring structure can be defined on $\Mot^{qpr}_{sch}(k)$ by
\[\lambda_{X}(z)\coloneqq(\zeta_X(-z))^{-1},\]
where $X$ is a reduced quasi-projective scheme. 

By Lemma \ref{qpr}, we first obtain a pre-$\lambda$-ring structure on $\Mot_{sch}(k)$. We extend the pre-$\lambda$-ring structure on $\Mot_{sch}(k)$ to the pre-$\lambda$-ring structure on $\Mot_{sch}(k)[\mathbb{L}^{-1}]$ by Proposition \ref{zeta}(ii) and Lemma \ref{prelambdaex}. Next we extend the pre-$\lambda$-ring structure to $\overline{\Mot}_{sch}(k)$ by continuity. Finally, by Proposition \ref{completioniso}, we have an isomorphism 
$\overline{\Mot}(k)\simeq\overline{\Mot}_{sch}(k).$ 

Explicitly, any class $A\in\overline{\Mot}(k)$ can be written as the limit of a sequence $([Y_i]-[Z_i])/\mathbb{L}^{n_i}$, where $Y_i$ and $Z_i$ are reduced quasi-projective schemes. Then we have
\[\lambda_A(z)=\lim_{i\to\infty}\frac{\zeta_{Z_i}(-\mathbb{L}^{-n_i}z)}{\zeta_{Y_i}(-\mathbb{L}^{-n_i}z)},\]
which proves uniqueness.
\end{proof}
\begin{Rem}
We are using the opposite pre-$\lambda$-ring structure here because for the pre-$\lambda$-ring structure in \cite[Section 1.3]{FSS1}, we have $\lambda(1)=(1-z)^{-1}$. Thus it admits no non-trivial homomorphism to a $\lambda$-ring (see Section~\ref{section260}).
\end{Rem}
\subsection{Stack motivic functions}\label{section26}
For any stack $\mathcal{X}$ locally of finite type over a field $k$, we will define an abelian group $\Mot(\mathcal{X})$ of motivic functions on $\mathcal{X}$. When $\mathcal{X}=\Spec k$, we recover $\Mot(k)$.

Following the definitions in \cite[Section 2]{FSS1}, we define first the group of motivic functions for stacks of finite type except that we add a relation for the surjective and universally injective morphisms.
\begin{Def}
Let $\mathcal{X}$ be an Artin stack of finite type over $k$. The abelian group $\Mot(\mathcal{X})$ is the group generated by the isomorphism classes of finite type morphisms $\mathcal{Y}\to\mathcal{X}$ modulo relations
\begin{enumerate}[(i)]
\item $[\mathcal{Y}_1\to\mathcal{X}]=[\mathcal{Y}_2\to\mathcal{X}]+[(\mathcal{Y}_1-\mathcal{Y}_2)\to\mathcal{X}]$ whenever $\mathcal{Y}_2$ is a closed substack of~$\mathcal{Y}_1$.
\item $[\mathcal{Y}_2\to\mathcal{X}]=[\mathcal{Y}_1\to\mathcal{X}]$ if there is a surjective and universally injective morphism $\mathcal{Y}_2\to\mathcal{Y}_1$ of stacks over $\mathcal{X}$.
\item If $\pi\colon \mathcal{Y}_1\to\mathcal{X}$ is a morphism of finite type and $\psi\colon \mathcal{Y}_2\to\mathcal{Y}_1$ is a vector bundle of rank $r$, then
\[[\mathcal{Y}_2\xrightarrow{\pi\circ\psi}\mathcal{X}]=[\mathcal{Y}_1\times_k\mathbb{A}^r_k\xrightarrow{\pi\circ p_1}\mathcal{X}].\]
\end{enumerate}
We call $\Mot(\mathcal{X})$ the \textit{group of stack motivic functions} on $\mathcal{X}$. We write $[\mathcal{Y}]$ instead of  $[\mathcal{Y}\to\Spec k]$. Note that $\Mot(\Spec k)$ coincides with $\Mot(k)$ defined in Definition~\ref{motstacks}.
\end{Def}
We define the pullback and pushforward homomorphisms for stacks of finite type. Similarly to the content of \cite[Section~2.1]{FSS1}, for a morphism $f\colon \mathcal{X}\to\mathcal{Y}$ of stacks of finite type possibly over different fields, we define the pullback homomorphism $f^{\ast}\colon \Mot(\mathcal{Y})\to\Mot(\mathcal{X})$ by
\[f^{\ast}([\mathcal{S}\to\mathcal{Y}])=[\mathcal{S}\times_{\mathcal{Y}}\mathcal{X}\to\mathcal{X}].
\]
The pullback homomorphism is well-defined because the surjective and universally injective morphisms are stable under base change, see \cite[Tag03MH and Tag03MW]{SP}.

For a morphism $f\colon \mathcal{X}\to\mathcal{Y}$ of finite type, we also define the pushforward homomorphism $f_{!}\colon \Mot(\mathcal{X})\to\Mot(\mathcal{Y})$ by
\[
f_{!}([\pi\colon \mathcal{S}\to\mathcal{X}])=[f\circ\pi\colon \mathcal{S}\to\mathcal{Y}].
\]
The pullbacks and pushforwards satisfy the same properties as in \cite[Section~2.1]{FSS1}. For $k$-stacks $\mathcal{X}$ and $\mathcal{Y}$ of finite type, we have an external product $\boxtimes\colon \Mot(\mathcal{X})\otimes_{\mathbb{Z}}\Mot(\mathcal{Y})\to\Mot(\mathcal{X}\times_k\mathcal{Y})$.

To define the group of motivic functions for a stack $\mathcal{X}$ locally of finite type, first we denote by $\Opf(\mathcal{X})$ the set of finite type open substacks $\mathcal{U}\subset\mathcal{X}$ ordered by inclusion. 
\begin{Def}
Let $\mathcal{X}$ be a stack locally of finite type. We define
\[\Mot(\mathcal{X})\coloneqq\varprojlim\Mot(\mathcal{U}),\]
where the limit is taken over the partially ordered set $\Opf(\mathcal{X})$. 
\end{Def}
If $\mathcal{X}\to\mathcal{Y}$ is a morphism of finite type, we write $[\mathcal{X}\to\mathcal{Y}]\in\Mot(\mathcal{Y})$ for the inverse system given by $\mathcal{U}\mapsto\mathcal{X}\times_{\mathcal{Y}}\mathcal{U}$.

Notice that the pullback extends to any morphism of stacks, and the pushforward extends to any morphism of finite type. Set
\begin{equation}\label{fin}
\Mot^{fin}(\mathcal{X})\coloneqq\cup_{\mathcal{U}\in\Opf(\mathcal{X})}(j_{\mathcal{U}})_!\left(\Mot(\mathcal{U})\right)\subset\Mot(\mathcal{X}),
\end{equation}
where $j_{\mathcal{U}}\colon \mathcal{U}\to\mathcal{X}$ is the open immersion. This is the group of ``motivic functions with finite support''. Given a morphism of finite type, the pushforward and pullback homomorphisms preserve $\Mot^{fin}$.

Let $\mathcal{X}$ be an Artin stack of finite type. We define the completion $\overline{\Mot}(\mathcal{X})$ with respect to the relative dimension filtration and call it \emph{the completed group of stack motivic functions}. In particular, $\overline{\Mot}(\Spec k)$ coincides with $\overline{\Mot}(k).$ 

If $\mathcal{X}$ is an Artin stack locally of finite type, we define
\[\overline{\Mot}(\mathcal{X})\coloneqq\varprojlim\overline{\Mot}(\mathcal{U}),\]
where the limit is taken over the partially ordered set $\Opf(\mathcal{X})$. Note that the pullback $f^{\ast}$ and pushforward $f_!$ can be extended to the completions. The group $\overline{\Mot}^{fin}(\mathcal{X})$ is defined similarly to (\ref{fin}).
\subsubsection{Bilinear form}\label{section270} Let $\mathcal{Y}$ be a $k$-stack of finite type. If $\mathcal{X}_1$ and $\mathcal{X}_2$ are of finite type over $\mathcal{Y}$, we set
\[([\mathcal{X}_1\to\mathcal{Y}]\vert[\mathcal{X}_2\to\mathcal{Y}])\coloneqq[\mathcal{X}_1\times_{\mathcal{Y}}\mathcal{X}_2].\]
We get a symmetric bilinear form $\Mot(\mathcal{Y})\otimes\Mot(\mathcal{Y})\to\Mot(k)$ by extending this by bilinearity. This can also be extended to a symmetric form $\overline{\Mot}(\mathcal{Y})\otimes\overline{\Mot}(\mathcal{Y})\to\overline{\Mot}(k)$ by continuity.

Let $\mathcal{Y}$ be a $k$-stack locally of finite type. Given $A\in\overline{\Mot}^{fin}(\mathcal{Y})$, $B\in\overline{\Mot}(\mathcal{Y})$, we write $A=j_!A_{\mathcal{V}}$, where $\mathcal{V}\in\Opf(\mathcal{Y})$ with open immersion $j\colon \mathcal{V}\to\mathcal{Y}$. Let $B$ be given by an inverse system $\mathcal{U}\mapsto B_{\mathcal{U}}$, where $\mathcal{U}$ ranges over $\Opf(\mathcal{Y})$. Set $(A\vert B)\coloneqq(A_{\mathcal{V}}\vert B_{\mathcal{V}})$, which is independent of the choice of $\mathcal{V}$. Now we have a continuous bilinear form 
\[(\bullet\vert\bullet)\colon \overline{\Mot}^{fin}(\mathcal{Y})\otimes\overline{\Mot}(\mathcal{Y})\to\overline{\Mot}(k).\]
We will get a symmetric bilinear form if we restrict to $\overline{\Mot}^{fin}(\mathcal{Y})\otimes\overline{\Mot}^{fin}(\mathcal{Y})$.
\subsection{Motivic classes over the perfection of the field}
We want to give a theorem comparing the motivic classes over a field $k$ and its perfection.

Recall that the perfection of a field $k$ is $k^{perf}\coloneqq\colim_n k^{1/p^n}$ where $k^{1/p^n}$ consists of $p^n$-th roots of elements in $k$.
\begin{Th}\label{perfectfield}
Let $f:\Spec k^{perf}\to\Spec k$ be the natural morphism. Then the morphism \[f^\ast:\Mot(k)\xrightarrow{\sim}\Mot(k^{perf})\] is an isomorphism of rings.
\end{Th}
\begin{proof}
By the definition of $k^{perf}$, we have
\[\Mot(k^{perf})=\colim_n\Mot(k^{1/p^n}).\]
We have a natural morphism $g:\Spec k^{1/p}\to\Spec k$. Now it suffices to prove that the induced morphisms $g^\ast$ and $g_\ast$ are inverse to each other
\[g^\ast:\Mot(k)\xrightarrow{\sim}\Mot(k^{1/p}):g_\ast.\]
By Theorem \ref{motequal}, it suffices to prove a similar statement with $\Mot(k)$ replaced by $\Mot_{sch}(k)$. One composition takes the class of a $k^{1/p}$-scheme $X$ to the class of its Frobenius twist $X^{(p)}\coloneqq X\times_kk^{1/p}$ over $k$. We have a morphism $X^{(p)}\to X$ that is surjective and universally injective because it is a base change of $\Spec k^{1/p}\to\Spec k$.

The other composition takes the class of a $k$-scheme $Y$ to the class of $Y\times_kk^{1/p}$ viewed as a $k$-scheme. This composition is the identity for a similar reason.
\end{proof}
\subsection{The universal $\lambda$-ring quotient}\label{section260}
We recall the definition of the universal quotient of $\overline{\Mot}(k)$ that is a $\lambda$-ring, see \cite{LL} and \cite[Section 8]{FSS3} for more details.

Let $R$ be a pre-$\lambda$-ring (see Definition \ref{prelambda}), and write $\lambda\colon R\to(1+zR[[z]])^{\times}$ as $\lambda=\sum_{i=0}^\infty\lambda_iz^i$. Recall that $R$ is a $\lambda$-ring if $\lambda(1)=1+z$ and for all $x,y\in R$ and $m,n\in\mathbb{Z}_{\ge0}$, $\lambda_n(xy)$ and $\lambda_n(\lambda_m(x))$ can be expressed in terms of $\lambda_i(x)$ and $\lambda_j(y)$ using the ring operations in a standard way (see \cite[Expos\'e~V, D\'efinition 2.4,~(2.4.1)]{SGA6} and \cite[Chapter I., Section 1]{Knu}).

We denote by $\overline{\Mot}^{\lambda-ring}(k)$ the universal quotient of $\overline{\Mot}(k)$ that is a $\lambda$-ring, see \cite{LL}.
\begin{Rem}
We don't know whether the following two morphisms are injective:\[\Mot(k)\to\overline{\Mot}(k)\to\overline{\Mot}^{\lambda-ring}(k).\]
\end{Rem}
\section{Counting measure for the ring of motivic classes}\label{count}
In this section we define a counting measure for the stacks of finite type over finite fields and show that the volume of the groupoid only depends on the image in the universal $\lambda$-ring quotient.
\subsection{Counting measure for the universal $\lambda$-ring quotient}
Recall the definition of the volume of a groupoid.
\begin{Def}
Let $\mathcal{G}$ be a groupoid with finite number of isomorphism classes. We define the volume of a groupoid $\mathcal{G}$ as the weighted sum:
\[\vert\mathcal{G}\vert\coloneqq\sum_{G\in\mathcal{G}/\sim}\frac{1}{\vert\Aut(G)\vert},\]
where the sum is over the set of isomorphism classes in $\mathcal{G}$.
\end{Def}
Fix a prime power $q$. If $\mathcal{X}$ is a stack of finite type over the finite field $\mathbb{F}_q$, we denote by $\vert \mathcal{X}(\mathbb{F}_q)\vert$ the volume of the groupoid $\mathcal{X}(\mathbb{F}_q)$. Note that for a scheme $X$, $\vert X(\mathbb{F}_q)\vert$ coincides with the number of $\mathbb{F}_q$-points of $X$.
\begin{Th}\label{countingstack}
\begin{enumerate}[(i)]
    \item We have a unique ring homomorphism
\[\#\colon\Mot(\mathbb{F}_q)\to\mathbb{Q}\]
such that for every finite type $\mathbb{F}_q$-stack $\mathcal{X}$, $\#([\mathcal{X}])=\vert\mathcal{X}(\mathbb{F}_q)\vert$.
\item For any $A\in\Mot(\mathbb{F}_q)$, the counting measure $\#(A)$ depends only on the image of $A$ in $\overline{\Mot}^{\lambda-ring}(\mathbb{F}_q)$.
\end{enumerate}
\end{Th}
\begin{proof}[Proof of part (i)]
First we have a ring homomorphism
\[\Mot_{sch}(\mathbb{F}_q)\to\mathbb{Z}\] sending $[X]$ to $\vert X(\mathbb{F}_q)\vert$ by \cite[Proposition 7.26]{Mus}.

Then we can extend the ring homomorphism as follows:
\[\#\colon \Mot_{sch}(\mathbb{F}_q)[\mathbb{L}^{-1},(\mathbb{L}^{i}-1)^{-1}\vert i>0]\to\mathbb{Q}\]
where we send $\mathbb{L}^{-1}$ to $q^{-1}$ and $(\mathbb{L}^{i}-1)^{-1}$ to $(q^{i}-1)^{-1}$.

Composing Theorem \ref{motequal} and \cite[Lemma 3.5.1, Propositions 3.5.6 and 3.5.9]{Kre}, it suffices to show that for a scheme $X$, 
\[\vert \left(X/\GL(n)\right)(\mathbb{F}_q)\vert=\frac{\#([X])}{\#([\GL(n)])}.\]
By the proof of Theorem \ref{motequal} and the definition of $\#$,
\[\#([X/\GL(n)])=\frac{\#([X])}{\#([\GL(n)])}.\]
Applying \cite[Lemma 2.5.1]{Beh}, we finally have
\[\vert \left(X/\GL(n)\right)(\mathbb{F}_q)\vert=\frac{\#([X])}{\#([\GL(n)])}=\#([X/\GL(n)]).\]
The uniqueness is obvious. 
\end{proof}
This ring homomorphism is called the \emph{counting measure}.
\begin{Rem}
Our proof is inspired by \cite[Section 5]{BD}, where for a subring $\overline{\Mot}_{conv}(\mathbb{F}_q)$ of $\overline{\Mot}(\mathbb{F}_q)$ consisting of the so-called convergent elements, the authors defined a counting measure 
\[\#\colon  \overline{\Mot}_{conv}(\mathbb{F}_q)\to\mathbb{\mathbb{C}}\]
which is a ring homomorphism. Note that $\#$ is not continuous, for example, the sequence $q^n/\mathbb{L}^n$ converges to zero in $\overline{\Mot}(\mathbb{F}_q)$, but the counting measure converges to~$1$.
\end{Rem}
To prove the second part of the above theorem, first we need to define the following $\lambda$-ring $R_q$.
\subsection{$\lambda$-ring $R_q$}\label{lambdaring}
We first recall the definition of Weil number: $\alpha\in\overline{\mathbb{Q}}$ is called a $q$-Weil number of weight $i$, if for all embeddings $\sigma\colon\overline{\mathbb{Q}}\hookrightarrow\mathbb{C}$, we have $\vert\sigma(\alpha)\vert=q^{i/2}$.

We define the following abelian symmetric monoidal category
\[J_q=\{(V,\phi)\colon V\text{ is a finite dimensional }\overline{\mathbb{Q}}_{\ell}\text{-vector space, }\phi\colon V\to V \text{ is an}\]
\[\text{endomorphism where the eigenvalues of }\phi\text{ are }q\text{-Weil numbers}\}.\]
We also define its full subcategory
\[J_{q,i}=\{(V,\phi)\colon V\text{ is a finite dimensional }\overline{\mathbb{Q}}_{\ell}\text{-vector space, }\phi\colon V\to V \text{ is an}\]
\[\text{endomorphism where the eigenvalues of }\phi\text{ are }q\text{-Weil numbers of weight }i\}.\]
Set $R_q\coloneqq K[J_q]\otimes\mathbb{Q}$ where $K[J_q]$ denotes the Grothendieck ring of $J_q$. In other words, ~$R_q$ is a $\mathbb{Q}$-vector space generated by the
isomorphism classes of objects in $J_q$ modulo the relation 
\[[V,\phi]=[V_1,\phi_1]+[V_2,\phi_2]\] 
for all exact sequences 
\[0\to(V_1,\phi_1)\to(V,\phi)\to(V_2,\phi_2)\to0.\]
It has a $\mathbb{Q}$-algebra structure given by
\[[V_1,\phi_1]\cdot[V_2,\phi_2]\coloneqq[V_1\otimes V_2,\phi_1\otimes\phi_2].\]
We have $R_q=\bigoplus_iR_{q,i}$ where $R_{q,i}\coloneqq K[J_{q,i}]\otimes\mathbb{Q}$. Thus $R_q$ is a graded ring.

For $m\ge0$, we define $F^mR_q\coloneqq\bigoplus_{i\le-m}R_{q,i}$. This gives a ring filtration and we define the completed ring $\overline{R}_q$ as the completion of $R_q$ with respect to this filtration, i.e.,
\[\overline{R}_q\coloneqq\left\{\sum_{-\infty}^{j=N}[V_j,\phi_j]\colon [V_j,\phi_j]\in R_{q,j}.\right\}\]
The isomorphism classes of the simple
objects in $J_q$ are $\mathcal{E}_\lambda\coloneqq(\overline{\mathbb{Q}}_{\ell},\lambda)$, where $\lambda\in\overline{\mathbb{Q}}_{\ell}=\End(\overline{\mathbb{Q}}_{\ell})$ is such that $\lambda$ is a $q$-Weil number. Thus, $R_q$ as
a $\mathbb{Q}$-vector space has a basis $[\mathcal{E}_\lambda]$ and the product is $[\mathcal{E}_\lambda][\mathcal{E}_\mu] = [\mathcal{E}_{\lambda\mu}]$.

Since $\mathbb{Q}\subset R_q$, by \cite[Section 8.1]{FSS3}, we can define the $\lambda$-ring structure by specifying the $\psi$-operations. We define the $\psi$-operations to be the homomorphisms $\psi_n\colon R_q\to R_q$ such that for the basis $[\mathcal{E}_\lambda]\in R_q$, $\psi_n([\mathcal{E}_\lambda])=[\mathcal{E}_{\lambda^n}]$. Now we have the following properties $\psi_0=1,\psi_1=\Id_{R_q}$, and for all $n,m\in\mathbb{Z}_{\ge0}$ we have $\psi_{nm}=\psi_{n}\circ\psi_{m}$.  This gives us a well-defined $\lambda$-ring structure. 

We will need the following lemmas.
\begin{Lemma}\label{superlambda}
For $\mathcal{V}_{\pm}\in J_q$ we have
\[\lambda([\mathcal{V}_+]-[\mathcal{V}_-])=\left(\sum _{i=0}^{\dim \mathcal{V}_+}[\wedge^i\mathcal{V}_+]z^i\right)\left(\sum_{i=0}^{\infty}[Sym^i\mathcal{V}_-](-z)^i\right).\]
\end{Lemma}
\begin{proof}
(i) We first show that
\[\lambda([\mathcal{V}_+])=\sum _{i=0}^{\dim \mathcal{V}_+}[\wedge^i\mathcal{V}_+]z^i.\]
For an exact sequence $0\to \mathcal{V}_1\to \mathcal{V}_+\to \mathcal{V}_2\to0$, there is a filtration
\[0=\mathcal{F}^0\subseteq\mathcal{F}^1\subseteq\mathcal{F}^2\subseteq\cdots\subseteq \mathcal{F}^{i}=\wedge^{i}\mathcal{V}_+,\]
such that $\mathcal{F}^j/\mathcal{F}^{j-1}\simeq\wedge^j\mathcal{V}_1\otimes\wedge^{i-j}\mathcal{V}_2$. Thus,
\[[\wedge^i\mathcal{V}_+]=\bigoplus_{j=0}^i[\wedge^j\mathcal{V}_1]\cdot[\wedge^{i-j}\mathcal{V}_2].\]
It follows that it is enough to check the identity when $\mathcal{V}_+ =\mathcal{E}_\lambda$. By \cite[Lemma~8.1.1]{FSS3},
\[\lambda([\mathcal{E}_\lambda])=1+z[\mathcal{E}_\lambda]=\sum_{i=0}^1[\wedge^i\mathcal{E}_\lambda]z^i.\]
Similarly, 
\[\lambda(-[\mathcal{E}_\lambda])=\frac{1}{1+z[\mathcal{E}_\lambda]}=\sum_{i=0}^{\infty}[Sym^i\mathcal{E}_\lambda](-z)^i,\]
and this implies that
\[\lambda(-[\mathcal{V}_-])=\sum_{i=0}^{\infty}[Sym^i\mathcal{V}_-](-z)^i.\]
\end{proof}
We can re-write the above formula using super objects. Let $\mathcal{V}=\mathcal{V}_-\oplus \mathcal{V}_+$ be a super object of $J_q$, we define $[\mathcal{V}]\coloneqq[\mathcal{V}_+]-[\mathcal{V}_-]\in R_q$. We have
\[Sym^i\mathcal{V}\coloneqq\bigoplus_{j=0}^i\left(Sym^j\mathcal{V}_+\otimes\wedge^{i-j}\mathcal{V}_-\right).\]
Note that the parity of a summand is given by the parity of $i-j$.
\begin{Lemma}\label{Jq}
Let $\mathcal{V}$ be a super object of $J_q$, we have
\[\lambda(-[\mathcal{V}])=\sum_{i=0}^{\infty}[Sym^i\mathcal{V}](-z)^i.\]
\end{Lemma}
\begin{proof}
Write $\mathcal{V}=\mathcal{V}_-\oplus \mathcal{V}_+$. By Lemma \ref{superlambda},
\newline
\leavevmode\vadjust{\kern\dimexpr-\abovedisplayskip-\baselineskip\relax}
\begin{align*}
\lambda(-[\mathcal{V}])&=\lambda([\mathcal{V}_-]-[\mathcal{V}_+])\\
&=\left(\sum_{i=0}^{\infty}[Sym^i\mathcal{V}_+](-z)^i\right)\left(\sum _{i=0}^{\dim \mathcal{V}_+}[\wedge^i\mathcal{V}_-]z^i\right)\\
&=\sum_{i=0}^{\infty}[Sym^i\mathcal{V}](-z)^i.
\end{align*}
The last equality holds because $[Sym^i\mathcal{V}]=\sum_j(-1)^{i-j}[Sym^j\mathcal{V}_+][\wedge^{i-j}\mathcal{V}_-]$.
\end{proof}
\subsection{\'Etale realization of motivic classes}Now we construct the \'etale realization of motivic classes and give the proof of the second part of Theorem \ref{countingstack}.

Let $X$ be a reduced quasi-projective scheme over $\mathbb{F}_q$. We define 
\[\overline{X}\coloneqq X\times_{\mathbb{F}_q}\Spec\overline{\mathbb{F}}_q\] where $\overline{\mathbb{F}}_q$ is a fixed algebraic closure of $\mathbb{F}_q$ and denote by $F_q$ the Frobenius action on the $\ell$-adic cohomology.
\begin{Prop}\label{Hast}
There is a unique continuous ring homomorphism \[H^{\ast}\colon \overline{\Mot}(\mathbb{F}_q)\to\overline{R}_q\] sending the class $[X]$ of a reduced quasi-projective scheme $X$ to the following class: 
\[\sum_{i}(-1)^i[H^i_c(\overline{X},\overline{\mathbb{Q}}_{\ell}),F_q]\in R_q.\]
\end{Prop}
First we need the following lemma.
\begin{Lemma}\label{homeomorphismlemma}
Let $f\colon  X\to Y$ be a finite type morphism of Noetherian
schemes. Assume that $f$ is surjective and universally injective. Then
it is a universal homeomorphism up to stratification. More precisely,
there is a locally closed stratification $Y=\bigsqcup_i Y_i$ such that the induced
morphisms $f^{-1}(Y_i)\to Y_i$ are universal homeomorphisms.
\end{Lemma}
\begin{proof}
By Noetherian induction, it is enough to show that there is an open $U\subset Y$ such that $f^{-1}(U)\to U$ is a universal homeomorphism. Since the morphism $f$ is quasi-finite,
after shrinking $Y$, we may assume that $f$ is finite (every quasi-finite morphism is
finite over an open subset of the target, see \cite[Tag03I1]{SP}). We can further
assume that $Y=\Spec A$ is affine. Then $X=\Spec B$ is affine as well and $f$
is induced by an integral morphism of rings $A\to B$.

Finally by \cite[Tag04DF]{SP}, we know an integral, universally injective and surjective morphism is a universal homeomorphism.  
\end{proof}
\begin{proof}[Proof of Proposition \ref{Hast}]
\emph{Step 1.} Recall the definition of $\Mot_{sch}(k)$ in Definition \ref{motscheme}. We claim that there is a ring homomorphism $H^\ast\colon \Mot_{sch}(\mathbb{F}_q)\to R_q$ such that
\[H^\ast([X])=\sum_{i}(-1)^i[H^i_c(\overline{X},\overline{\mathbb{Q}}_{\ell}),F_q].\]
Recall that the absolute values of the eigenvalues of $F_q$ are of the form $q^{i/2}, i\in\mathbb{Z}$, see \cite[Corollaire 3.3.4]{Del3}.
We just need to check the following properties:
\begin{enumerate}[(i)]
    \item $H^\ast([X])=H^\ast([Y])+H^{\ast}([X-Y])$ where $Y$ is a closed subscheme of $X$ over $\mathbb{F}_q$.
\item $H^\ast([X])=H^\ast([Y])$ if there is a surjective and universally injective morphism $X\to Y$ of schemes over $\mathbb{F}_q$.
\item $H^{\ast}([X\times_{\mathbb{F}_q} Y])=H^\ast([X])\cdot H^\ast([Y])$ where $X$, $Y$ are schemes over $\mathbb{F}_q$.
\end{enumerate}
The first relation holds by the long exact sequence of the cohomology, see \cite[Expos\'e XVII, (5.1.16.3)]{Del}. For the second relation, by Lemma \ref{homeomorphismlemma}, we know there is a stratification by universal homeomorphisms. Then by \cite[Tag03SI]{SP}, a universal homeomorphism induces an isomorphism of \'etale sites, we are done. Finally the third relation holds by K\"unneth theorem, see \cite[Expos\'e XVII, Th\'eor\`eme 5.4.3]{Del} and \cite[Page 24]{Kat}.

\emph{Step 2.} Note that $H^\ast(\mathbb{L})=[\overline{\mathbb{Q}}_{\ell},q]$. Since $\mathcal{L}\coloneqq[\overline{\mathbb{Q}}_{\ell},q]$ is invertible in $R_q$, we can extend $H^{\ast}$ to $\Mot_{sch}(\mathbb{F}_q)[\mathbb{L}^{-1}]$. We know that the eigenvalues of the Frobenius acting on $i$-th cohomology group are $q$-Weil numbers of weights bounded by $i$ by \cite[Corollaire 3.3.4]{Del3}. Also $H^i_c(\overline{X},\overline{\mathbb{Q}}_{\ell})=0$ for $i>2\dim X$, then we have
$H^{\ast}([X])\in\bigoplus_{i\le 2\dim X}R_{q,i}$.

We see that the homomorphism 
\[H^\ast\colon \Mot_{sch}(\mathbb{F}_q)[\mathbb{L}^{-1}]\to R_q\] 
takes $F^m\Mot_{sch}(\mathbb{F}_q)[\mathcal{L}^{-1}]$ to $F^{2m}R_q$, so it is continuous and can be extended to the completion $\overline{\Mot}_{sch}(\mathbb{F}_q)$. Finally 
$\overline{\Mot}_{sch}(\mathbb{F}_q)\cong\overline{\Mot}(\mathbb{F}_q)$ by Proposition \ref{completioniso}.
\end{proof}
Recall that $\mathcal{L}\coloneqq[\overline{\mathbb{Q}}_{\ell},q]$ is invertible in $R_q$. We define a localization of $R_q$ as \[R_{q,fin-type}\coloneqq R_q[(\mathcal{L}^{i}-1)^{-1}\vert i>0].\]
Since $\mathcal{L}^i-1$ is invertible in $\overline{R}_q$, the homomorphism $R_q\to \overline{R}_q$ factors through $R_{q,fin-type}$. Note that $R_q\to\overline{R}_q$ is injective, so the induced morphism $R_{q,fin-type}\to\overline{R}_q$ is injective as well.

We also define a map $\tr\colon R_q\to\overline{\mathbb{Q}}_{\ell}\text{ given by }[V,\psi]\mapsto\tr\psi.$
We claim that it is a ring homomorphism. It suffices to check for the basis $[\mathcal{E}_\lambda], [\mathcal{E}_\mu]\in R_{q}$:
\[\tr\left([\mathcal{E}_\lambda][\mathcal{E}_\mu]\right)=\tr\left([\mathcal{E}_{\lambda\mu}]\right)=\lambda\mu=\tr\left([\mathcal{E}_\lambda]\right)\cdot\tr\left([\mathcal{E}_\mu]\right).\]
Note that the homomorphism $\tr$ takes $\mathcal{L}$ to $q$ and $\mathcal{L}^i-1$ to $q^i-1\neq 0$, so by the universal property of localizations, it extends to a ring homomorphism 
$\tr\colon R_{q,fin-type}\to\overline{\mathbb{Q}}_{\ell}.$
\begin{Prop}\label{countfintype}
We have the following commutative diagram:
\begin{equation}\label{Ql}
\begin{tikzcd}
\overline{\Mot}(\mathbb{F}_q)\arrow{r}{H^\ast} & \overline{R}_q \\
\Mot(\mathbb{F}_q)\arrow{u}\arrow{d}{\#}\arrow{r}&    R_{q,fin-type}\arrow{d}{\tr}\arrow[hookrightarrow]{u}\\
\mathbb{Q}\arrow[hookrightarrow]{r} & \overline{\mathbb{Q}}_{\ell}.
\end{tikzcd}
\end{equation}
Note that $\#$ is a homomorphism of Theorem \ref{countingstack}.
\end{Prop}
\begin{proof}
By construction, the image of 
\[\Mot(\mathbb{F}_q)\to\overline{\Mot}(\mathbb{F}_q)\xlongrightarrow{H^{\ast}}\overline{R}_q\]
is contained in $R_{q,fin-type}$. Thus the upper square of diagram (\ref{Ql}) commutes.

We show that the lower square commutes. Recall the Grothendieck-Lefschetz trace formula (see \cite[Chapitre 2, Th\'eor\`eme~3.2]{Del1}) for a scheme $X$ over~$\mathbb{F}_q$:
\[\vert X(\mathbb{F}_q)\vert=\sum_{i}(-1)^i\tr F_q\vert_{H^i_c(\overline{X},\overline{\mathbb{Q}}_{\ell})}.\]
Thus the counting measure factors through $R_q$, that is, we have a commutative diagram:
\[
\begin{tikzcd}
\Mot_{sch}(\mathbb{F}_q)\arrow{d}{\#}\arrow{r}{H^{\ast}}  & R_q\arrow{d}{\tr}\\
\mathbb{Z}\arrow[hookrightarrow]{r} & \overline{\mathbb{Q}}_{\ell}.
\end{tikzcd}
\]
Recall that $\tr(\mathcal{L})=q$, then by Theorem \ref{motequal} and the universal property of localization, we see that the lower square of diagram (\ref{Ql}) commutes.
\end{proof}
Recall that in Section \ref{lambdaring} we defined a $\lambda$-ring structure on $R_q$. The $\lambda$-ring structure on $R_q$ can be extended to $\overline{R}_q$ by continuity. We also know that $\overline{\Mot}(\mathbb{F}_q)$ is a pre-$\lambda$-ring by Theorem \ref{prestructure}.
\begin{Prop}\label{conj}
The homomorphism $H^{\ast}\colon \overline{\Mot}(\mathbb{F}_q)\to\overline{R}_q$ in Proposition \ref{Hast} is a homomorphism of pre-$\lambda$-rings.
\end{Prop}
\begin{proof}
Let $X$ be a reduced quasi-projective scheme over the finite field $\mathbb{F}_q$. Denote by $\lambda_H$ the $\lambda$-ring structure on $\overline{R}_{q}$. First we check that
\[\lambda_H(-H^\ast([X]))=H^{\ast}(\lambda(-[X])).\]
Set $H^\ast(X)\coloneqq\bigoplus_i (H^i_c(\overline{X},\overline{\mathbb{Q}}_{\ell}),F_q)$ and view it as a super object of $J_q$, so we have $[H^\ast(X)]\in R_q$.

By Lemma \ref{Jq},
\[\lambda_H(-H^\ast([X]))=\sum_i[Sym^i(H^\ast(X))](-t)^i.\]
On the other hand, we have
\[H^\ast(\lambda(-[X]))=\sum_iH^\ast([Sym^iX])(-t)^i=\sum_i[H^\ast(Sym^iX)](-t)^i.\]
However by Deligne's theorem (see \cite[Expos\'e XVII, Th\'eor\`eme 5.5.21]{Del}),
\[Sym^i(H^\ast(X))\simeq H^{\ast}(Sym^iX).\]
Since $\overline{\Mot}(\mathbb{F}_q)$ is topologically generated by classes $\mathbb{L}^{-i}[X]$, the statement follows.
\end{proof}
\begin{proof}[Proof of part (ii) of Theorem \ref{countingstack}]
By Proposition \ref{conj}, $H^\ast$ factors as 
\[\overline{\Mot}(\mathbb{F}_q)\to\overline{\Mot}^{\lambda-ring}(\mathbb{F}_q)\xlongrightarrow{H^{\ast,\lambda-ring}}\overline{R}_q.\]
Combining the above with Proposition \ref{countfintype}, we get a commutative diagram:
\[\begin{tikzcd}
\overline{\Mot}(\mathbb{F}_q)\arrow{r} & \overline{\Mot}^{\lambda-ring}(\mathbb{F}_q)\arrow{r}{H^{\ast,\lambda-ring}}& \overline{R}_q \\
\Mot(\mathbb{F}_q)\arrow{u}\arrow{d}{\#}\arrow{rr} & & R_{q,fin-type}\arrow{d}{\tr}\arrow[hookrightarrow]{u}\\
\mathbb{Q}\arrow[hookrightarrow]{rr} & & \overline{\mathbb{Q}}_{\ell}.
\end{tikzcd}
\]
Now the statement follows by diagram chasing.
\end{proof}
\section{Power structures}\label{realsection3}
Our goal in this section is to define a power structure on $\overline{\Mot}(k)$ and interpret certain powers as configuration spaces, see Proposition \ref{powerstructure} and Theorem \ref{power}. These are similar to \cite[Definition 1]{GZLMH2} and \cite[Lemma 6]{BM} respectively.
\subsection{From pre-$\lambda$-rings to plethystic exponents and plethystic logarithms}\label{section23}
In this section, for an arbitrary pre-$\lambda$-ring, as a generalization of \cite[Remark 2, Page 55]{GZLMH1} we define the multivariable generalizations of the plethystic exponents and plethystic logarithms.

Let $R$ be a pre-$\lambda$-ring with the pre-$\lambda$-ring structure $a\mapsto\lambda_a(z)$. We denote by $R[[z_1,z_2,\ldots,z_n]]^{+}$ the ideal of power series with vanishing constant term and denote by $(1+R[[z_1,z_2,\ldots,z_n]]^{+})^{\times}$ the multiplicative group of series with constant term equal to $1$, we define the plethystic exponent \[\Exp\colon R[[z_1,z_2,\ldots,z_n]]^{+}\to(1+R[[z_1,z_2,\ldots,z_n]]^{+})^{\times}\]
by
\[\Exp\left(\sum_{r_1,r_2,\ldots,r_n}A_{r_1,r_2,\ldots,r_n}\prod_{i=1}^nz_i^{r_i}\right)\coloneqq\prod_{r_1,r_2,\ldots,r_n}\lambda^{-1}_{A_{r_1,r_2,\ldots,r_n}}(-\prod_{i=1}^nz_i^{r_i}).\]
\begin{Lemma}
$\Exp$ is an isomorphism of abelian groups.
\end{Lemma}
\begin{proof}
\emph{Step 1.} For $n>0$, let $A_m$ denote the ideal of $R[[z_1,z_2,\ldots,z_n]]$ generated by monomials of degree at least $m$, then $R[[z_1,z_2,\ldots,z_n]]^{+}=A_1$ and $\Exp(A_m)\subset 1+A_m$. Thus, $\Exp$ induces a homomorphism
\[A_m/A_{m+1}\to(1+A_m)^{\times}/(1+A_{m+1})^{\times}.\]
It is enough to show that this is an isomorphism for any $m$.

\emph{Step 2.} Note that for
\begin{align*}
A&=\sum_{\sum r_i= m}A_{r_1,\ldots,r_n}\prod_{i=1}^nz_i^{r_i}+\text{monomials of degree higher than } m,
\intertext{we have}
\Exp\left(A\right)&=\left(\prod_{\sum r_i= m}\lambda_{A_{r_1,\ldots,r_n}}(\prod_{i=1}^nz_i^{r_i})\right)\left(1+\text{monomials of degree higher than } m\right)\\
&=\prod_{\sum r_i= m}(1+A_{r_1,\ldots,r_n}\prod_{i=1}^nz_i^{r_i}+\cdots)\left(1+\text{monomials of degree higher than } m\right)
\\&=1+\sum_{\sum r_i= m}A_{r_1,\ldots,r_n}\prod_{i=1}^nz_i^{r_i}+\text{monomials of degree higher than } m.
\end{align*}
Now the statement is clear.
\end{proof}
Since $\Exp$ is an isomorphism, we have the inverse isomorphism, which we call plethystic logarithm:
$\Log\coloneqq\Exp^{-1}$. In the next proposition, we verify the basic properties for them.
\begin{Prop}\label{exp}
With the notation $A(z)\coloneqq A(z_1,z_2,\ldots,z_n)$, we have:
\begin{align*}
\Exp(A(z)+B(z))&=\Exp(A(z))\Exp(B(z));\\ \Log(M(z)\cdot N(z))&=\Log(M(z))+\Log(N(z)).   
\end{align*}
\end{Prop}
\begin{proof}
The first property follows from the definitions of $\Exp$ and of pre-$\lambda$-ring structures. The second property holds since $\Log$ is the inverse of $\Exp$.
\end{proof}
\subsection{Definition of the power structure}\label{section24}
In this section, we give the definition of the power structure and explain how to obtain a power structure from a pre-$\lambda$-ring structure. This is a generalization of \cite[Definition 1]{GZLMH2} to the multivariable case.
\begin{Def}\label{pow}
A power structure over a ring $R$ is a map 
\[(1+ R[[z_1,z_2,\ldots,z_n]]^+)\times R\to 1+R[[z_1,z_2,\ldots,z_n]]^+\]
\[ \quad\quad(A(z),m)\mapsto(A(z))^m\]
which possesses the following properties:
\begin{enumerate}[(i)]
\item $(A(z))^0=1$.
\item $(A(z))^1=A(z)$.
\item $(A(z)\cdot B(z))^m=(A(z))^m\cdot(B(z))^m$.
\item $(A(z))^{m_1+m_2}=(A(z))^{m_1}\cdot(A(z))^{m_2}$.
\item $(A(z))^{m_1m_2}=((A(z))^{m_2})^{m_1}$.
\item $(1+\sum_{i}z_i)^m=1+\sum_{i}mz_i+\text{terms of degree at least }2$.
\item $(A(I(z)))^m=(A(z))^m\vert_{z\mapsto I(z)}$ where \[I(z_1,z_2,\ldots,z_n)=\left(\prod_{i=1}^nz_i^{r_{i,1}},\ldots,\prod_{i=1}^nz_i^{r_{i,n}}\right)\text{ with }r_{i,j}\text{ being nonnegative integers}.\]
\end{enumerate}
\end{Def}
\begin{Prop}\label{powerstructure}
Let $R$ be a pre-$\lambda$-ring and let $\Exp\colon R[[z_1,z_2,\ldots,z_n]]^{+}\to(1+R[[z_1,z_2,\ldots,z_n]]^{+})^{\times}$ and $\Log\colon (1+R[[z_1,z_2,\ldots,z_n]]^{+})^{\times}\to R[[z_1,z_2,\ldots,z_n]]^{+}$ be the corresponding plethystic operations (see Section \ref{section23}). Then 
\[(A(z),m)\mapsto\Exp(m\Log(A(z)))\] gives a power structure on $R$.
\end{Prop}
\begin{proof}
The properties (i), (ii) and (v) follow from the definitions, while properties (iii) and (iv) follow from Proposition \ref{exp}. The property (vii) is true because $\Exp$ commutes with monomial substitutions. Now it remains to prove property (vi). It follows from the definition of $\Exp$ and of pre-$\lambda$-ring structure that
\begin{align*}
&\Exp\left(\sum_{i}A_iz_i+\text{terms of degree at least }2\right)\\
=&1+\sum_iA_iz_i+\text{terms of degree at least } 2. 
\end{align*}
It follows easily that 
\[\Log(1+\sum_{i}M_iz_i)=\sum_{i}M_iz_i+\text{terms of degree at least }2.\]
Finally we have
\begin{align*}
\Exp\left(m\Log\left(1+\sum_{i}z_i\right)\right)&=\Exp\left(m\left(\sum_{i}z_i+\text{terms of degree at least } 2\right)\right)\\
&=1+\sum_{i}mz_i+\text{terms of degree at least } 2.\end{align*}
It is indeed a power structure.
\end{proof}
\begin{Rem}
The power structures corresponding to different pre-$\lambda$-ring structures might coincide. As in \cite[Section 3]{GZLMH2}, the power structures defined by motivic zeta-function
\[\zeta_M(z)=1+[S^1M]\cdot z+[S^2M]\cdot z^2+\cdots,\]
and generating series of the configuration space $M_k=(M_k\backslash\Delta)/S_k$ where $M$ is a quasi-projective variety and $\Delta$ is the large diagonal
\[\phi_M(z)=1+[M]\cdot z+[M_1]\cdot z^2+\cdots,\]
are the same.
\end{Rem}
\subsection{Power structures on the rings of motivic classes of schemes and stacks}\label{section25}
In Section \ref{section23} and Section \ref{section24} we dealt with arbitrary pre-$\lambda$-ring structures. Now we apply this formalism to the pre-$\lambda$-ring structure on $\overline{\Mot}(k)$ defined in Section \ref{section22}. Thus, we obtain a power structure on the ring $\overline{\Mot}(k)$. We have a theorem that is a generalization of \cite[Lemma 3.8.2]{FSS1} to arbitrary characteristic.
\begin{Th}\label{power}
Assume that $M$ is a scheme, $\mathcal{A}_0=\Spec k$ and $\mathcal{A}_1,\mathcal{A}_2,\ldots,\mathcal{A}_n,\ldots$ are stacks. Put $A(z)=[\mathcal{A}_0]+[\mathcal{A}_1]z+[\mathcal{A}_2]z^2+\cdots+[\mathcal{A}_n]z^n+\ldots$. Then we have \[(A(z))^{[M]}=1+\sum_{k=1}^{\infty}\left\{\sum_{k_i\colon \sum ik_i=k}\left[\left(\left(M^{\sum_i k_i}\backslash\Delta\right)\times\prod_i\mathcal{A}_i^{k_i}\right)/\prod_i{S_{k_i}}\right]\right\}\cdot z^k,\]
where the power structure is given in Proposition \ref{powerstructure} and $\Delta$ is the ``large diagonal'' in $M^{\sum_i k_i}$ which consists of $\left(\sum_i k_i\right)$-tuples of points of $M$ with at least two coinciding ones.
\end{Th}
\begin{proof}
By the proof of \cite[Lemma 6]{BM}, we reduce the stack case to the case where~$\mathcal{A}_i$ are schemes. Writing $B(z)=\Log A(z)$, we just need to check that \[\Exp([M]B(z))=\Exp(B(z))^{[M]}.\] Now, by \cite[Remark 2, Page 55]{GZLMH1}, \[\Exp([M]B(z))=\prod_{k\ge1}(1-z^k)^{-[M]B(z)},\]
and one just needs to use the definition of power structure. Note that \cite{GZLMH1} works with $k=\mathbb{C}$. However, using our Lemma \ref{G-invariant}, we can extend their results to any field $k$.
\end{proof}
\begin{Rem}
Theorem \ref{power} shows that the power structure is effective in the following sense: if $M$ is a scheme and the coefficient of $A(z)$ are motivic classes of stacks, then the coefficients of $(A(z))^{[M]}$ are also motivic classes of stacks. This is not always true if $M$ is a stack and there is a counterexample in \cite[Statement~2]{GZLMH2}:
\[(1+z)^{[\BGL(1)]}=(1+z)^{1/(\mathbb{L}-1)}\]
is not effective since the coefficient at $z^2$ is equal to
\[\frac{-\mathbb{L}^3+\mathbb{L}^2+\mathbb{L}}{[\GL(2)]},\]
where the highest degree of the numerator has negative coefficient.
\end{Rem}
\section{Applications to stacks of Higgs bundles}\label{section4}
\subsection{Moduli stacks of vector bundles and HN-filtrations}\label{section31}
In this section, $k$ is a field of arbitrary characteristic and $X$ is a smooth projective geometrically connected curve over~$k$. Let $\mathcal{B}un_r$ be the stack of vector bundles of rank $r$ (over the fixed curve~$X$). It is an Artin stack locally of finite type, see \cite[Theorem 4.6.2.1]{LMB}. Let $\mathcal{B}un_{r,d}$ be its clopen subset classifying bundles of degree $d$. We call a vector bundle on $X$ semistable if for every non-zero subbundle $F\subset E$ we have
\[\frac{\deg F}{\rk F}\le\frac{\deg E}{\rk E}.\]
The semistability is preserved by field extensions. Indeed, let $k\subset K$ be a field extension, let $\overline{K}$ be an algebraic closure of $K$ and let $\overline{k}$ be the algebraic closure of $k$ in $\overline{K}$. According to \cite[Proposition 3]{Lan}, we have
\[E\text{ is semistable}\Leftrightarrow E_{\overline{k}}\text{ is semistable}\Leftrightarrow E_{\overline{K}}\text{ is semistable}\Leftrightarrow E_K\text{ is semistable}.\]
The number $\deg E/\rk E$ is called the slope of $E$. It is well known that each vector bundle $E$ on $X_K$ possesses a unique filtration 
\[0=E_0\subset E_1\subset\ldots\subset E_t=E.\]
such that for $i=1,\ldots, t$ the sheaf $E_i/E_{i-1}$ is a semistable vector bundle and for $i=1,\ldots, t-1$ the slope of $E_i/E_{i-1}$ is strictly greater than the slope of $E_{i+1}/E_i$ (see \cite[Section 1.3]{HN}). This filtration is called the \emph{Harder-Narasimhan filtration} (or HN-filtration for brevity) on $E$ and the sequence of slopes $\tau_1>\ldots>\tau_t$, where $\tau_i=\deg(E_i/E_{i-1})/\rk(E_i/E_{i-1})$, is called the HN-type of $E$. It follows from \cite[Proposition 3]{Lan} that the HN-filtrations and HN-types are compatible with field extensions.
\begin{Lemma}\label{HN}
\begin{enumerate}[(i)]
\item There is an open substack $\mathcal{B}un_{r,d}^{\ge\tau}\subset\mathcal{B}un_{r,d}$ classifying vector bundles whose HN-type ($\tau_1>\ldots>\tau_t$) satisfies $\tau_t\ge\tau$.
\item A vector bundle $E\in\mathcal{B}un_{r,d}(K)$ is in $\mathcal{B}un_{r,d}^{\ge\tau}(K)$ if and only if there is no surjective morphism of vector bundles $E\to F$ such that the slope of $F$ is less than $\tau$.
\item The stack $\mathcal{B}un_{r,d}^{\ge\tau}$ is of finite type.
\end{enumerate}
\end{Lemma}
\begin{proof}
See \cite[Lemma 3.2.1]{FSS1}. Note that it is assumed that $k$ has characteristic~$0$ but this is not used in the proof.
\end{proof}
We will be mostly interested in the stack $\mathcal{B}un_{r,d}^{+}\coloneqq\mathcal{B}un_{r,d}^{\ge 0}$. We will call such vector bundles ``HN-nonnegative''. Note that the tensorisation with a line bundle of degree $e$ gives an isomorphism $\mathcal{B}un_{r,d}^{+}\simeq\mathcal{B}un_{r,d+er}^{\ge e}$. By Lemma \ref{HN}, $E$ is HN-nonnegative if and only if there are no surjective morphisms $E\to F$, where $F$ is a vector bundle such that $\deg F<0$.
\subsection{Moduli stacks of Higgs bundles}\label{section41}
All the definitions and notations in this section follow \cite[Section 1.3]{FSS1}.
\begin{Def}
Fix a smooth projective geometrically connected curve $X$ over $k$. 
\begin{enumerate}[(i)]
\item A Higgs bundle on $X$ is a pair $(E,\Phi)$ where $E$ is a vector bundle on $X$ and $\Phi\colon E\to E\otimes\Omega_X$ is an $\mathcal{O}_X$-linear morphism from $E$ to $E$ ``twisted'' by the sheaf of differential $1$-forms $\Omega_X$.
\item The rank of the pair $(E,\Phi)$ is the rank of $E$, similarly the degree is the degree of $E$.
\end{enumerate}
\end{Def}
We use $\mathcal{H}iggs_{r,d}$ to denote the moduli stack of rank $r$ degree $d$ Higgs bundles on~$X$. This is an Artin stack locally of finite type over $k$ by a simple argument similar to the proof of \cite[Proposition 1]{Fed}. The forgetful morphism $\mathcal{H}iggs_{r,d}\to\mathcal{B}un_{r,d}$ sending $(E,\Phi)\to E$, is schematic and of finite type. Set
\[\mathcal{H}iggs^{+}_{r,d}\coloneqq\mathcal{H}iggs_{r,d}\times_{\mathcal{B}un_{r,d}}\mathcal{B}un^{+}_{r,d},\]
by Lemma \ref{HN}, it is an open substack of finite type of $\mathcal{H}iggs_{r,d}$.
\begin{Def}
The Higgs bundle $(E,\Phi)$ is called semistable if for any non-zero subbundle $F\subset E$ preserved by $\Phi$,
\[\frac{\deg F}{\rk F}\le\frac{\deg E}{\rk E}.\]
\end{Def}
Similarly to the proof of \cite[Proposition 3]{Lan} (see also Section \ref{section31}), this is an open condition compatible with field extensions, so we denote by $\mathcal{H}iggs^{ss}_{r,d}$ the open substack of semistable Higgs bundles. Furthermore, the stack $\mathcal{H}iggs^{ss}_{r,d}$ is of finite type by \cite[Lemma 3.3.2(ii)]{FSS1}.
\subsection{Motivic Hall algebras}\label{section330}
In this section, we keep the assumptions from the previous section, in particular: $X$ is a smooth projective geometrically connected curve of genus $g$ over $k$. In addition we assume that there is a divisor $D$ on $X$ defined over $k$ such that $\deg D=1$. In other words, the degrees of points on $X$ are coprime.

For any stack $\mathcal{X}$, we consider the ring
\[\overline{\Mot}(\mathcal{X})[\sqrt{\mathbb{L}}]\coloneqq\overline{\Mot}(\mathcal{X})[t]/(t^2-\mathbb{L}).\]
Let $\mathcal{C}oh_{r,d}$ be the moduli stack of coherent sheaves on $X$ with rank $r$ and degree $d$. This is a stack locally of finite type over $k$ by \cite[Th\'eor\`eme 4.6.2.1]{LMB}. Set
\[\mathcal{H}'\coloneqq\bigoplus_{r\ge 0,  d\in\mathbb{Z}}\overline{\Mot}(\mathcal{C}oh_{r,d})[\sqrt{\mathbb{L}}].\]
(Note that $\mathcal{C}oh_{0,d}=\emptyset$ if $d<0$.) Similarly to \cite[Section 5]{FSS1}, we define a structure of an alegrba on $\mathcal{H}'$. We let $\mathbb{Z}$ act on $\mathcal{H}'$ via
\[r\cdot f=\mathbb{L}^{(1-g)rr'}f
\text{ whenever }f\in\overline{\Mot}(\mathcal{C}oh_{r',d})[\sqrt{\mathbb{L}}].
\]
Now we define the motivic Hall algebra as
\[\mathcal{H}\coloneqq\mathcal{H}'\otimes_{\mathbb{Z}}\mathbb{Z}[\mathbb{Z}].\]
We have the corresponding comultiplication and bilinear form defined in the same way as in \cite[Section 5.3 and Section 5.4]{FSS1}.

Let us denote the group $\overline{\Mot}(\mathcal{X})$ defined in \cite[Section 2.1]{FSS1} without surjective and universally injective relation by $\overline{\Mot}^{naive}(\mathcal{X})$ and the corresponding motivic Hall algebra in \cite[Section 5.2]{FSS1} by $\mathcal{H}^{naive}$. Then we have the following well-defined homomorphisms
\[\overline{\Mot}^{naive}(\mathcal{X})\to\overline{\Mot}(\mathcal{X})\text{ and }\mathcal{H}^{naive}\to\mathcal{H}.\]

Applying these homomorphisms to the identities in \cite[Section 5]{FSS1}, we get
identities in the ``corrected'' motivic Hall algebra such as \cite[Lemma 5.6.1 and Proposition 5.5.3]{FSS1}. The identity of \cite[Proposition 5.7.1]{FSS1} is also true in our $\mathcal{H}$: the argument goes through in arbitrary characteristic except for the part where the power structure is used, and this is where our Theorem \ref{power} should be applied. After
this, the calculations of \cite[Section 6]{FSS1} carry over
to the case of finite characteristic. Thus for the remaining sections, we can use the results in \emph{loc.~cit.} 
\subsection{Motivic classes of the stacks of vector bundles with nilpotent endomorphisms}\label{section44}
The following definitions come from \cite[Section 1.3]{FSS1} (note that we slightly re-write the definition of $J_\lambda^{mot}(z)$ as in \cite[Section 3]{Mel2}). We define the ``normalized'' zeta-function $\tilde{\zeta}_X(z)\coloneqq z^{1-g}\zeta_X(z)$ where $g$ is the genus of the curve~$X$.

Recall the definition of the motivic $L$-function in Section \ref{section14}. Set
\[
J^{mot}_{\lambda}(z)\coloneqq\prod_{\square\in\lambda}\frac{L^{mot}_X(z^{a(\square)}\mathbb{L}^{-l(\square)-1})}{(1-z^{a(\square)}\mathbb{L}^{-l(\square)-1})(1-z^{a(\square)}\mathbb{L}^{-l(\square)})_{\neq0}}\in\overline{\Mot}(k)[[z]],\]
where the product is over all boxes of the Young diagram corresponding to the partition. Here the notation $(-)_{\neq0}$ means that we omit this term if it is zero. In particular for the empty Young diagram $\lambda$ we have $J^{mot}_\lambda(z)=1$.

Set
\[L^{mot}(z_n,\ldots,z_1)\coloneqq\frac{1}{\prod_{i<j}\tilde{\zeta}_X\left(\frac{z_i}{z_j}\right)}\sum_{\sigma\in S_n}\sigma\left\{\prod_{i<j}\tilde{\zeta}_X\left(\frac{z_i}{z_j}\right)\frac{1}{\prod_{i<n}\left(1-\mathbb{L}\frac{z_{i+1}}{z_i}\right)}\cdot\frac{1}{1-z_1}\right\}.
\]
For a partition $\lambda=1^{r_1}2^{r_2}\ldots t^{r_t}$ such that $\sum_i r_i=n$, set $r_{<i}=\sum_{k<i}r_k$ and denote by $\res_\lambda$ the iterated residue along
\[\begin{matrix}
\frac{z_n}{z_{n-1}}=\mathbb{L}^{-1}, &\frac{z_{n-1}}{z_{n-2}}=\mathbb{L}^{-1},&\ldots,&\frac{z_{2+r_{<t}}}{z_{1+r_{<t}}}=\mathbb{L}^{-1},\\
\vdots & \vdots & & \vdots\\
\frac{z_{r_1}}{z_{r_1-1}}=\mathbb{L}^{-1}, &\frac{z_{r_1-1}}{z_{r_1-2}}=\mathbb{L}^{-1},&\ldots,&\frac{z_2}{z_1}=\mathbb{L}^{-1}.
\end{matrix}\]
Set 
\[\tilde{H}^{mot}_{\lambda}(z_{1+r_{<t}},\ldots,z_{1+r_{<i}},\ldots,z_1)\coloneqq\res_{\lambda}\left[L^{mot}(z_n,\ldots,z_1)\prod^n_{j=1\atop j\notin\{r_{<i}\}}\frac{dz_j}{z_j}\right]
\]
and
\[H^{mot}_{\lambda}(z)\coloneqq\tilde{H}^{mot}_{\lambda}(z^t\mathbb{L}^{-r_{<t}},\ldots,z^i\mathbb{L}^{-r_{<i}},\ldots,z)\in\overline{\Mot}(k)[[z]].\]
The notion of residue is not obvious for rational functions with coefficients belonging to $\overline{\Mot}(k)$, but a precise definition is given in \cite[Remark 1.3.2]{FSS1}.

We define ``Schiffmann's generating function'' as
\begin{equation}\label{schiff}
\Omega_X^{Sch,mot}\coloneqq\sum_{\lambda}w^{\vert\lambda\vert}\mathbb{L}^{(g-1)\langle\lambda,\lambda\rangle}J^{mot}_\lambda(z)H^{mot}_\lambda(z)\in\overline{\Mot}(k)[[z,w]].
\end{equation}
Here the sum is over all partitions $\lambda=(\lambda_1\ge\lambda_2\ge\ldots\ge\lambda_l>0)$, $\langle\lambda,\lambda\rangle\coloneqq\sum_i(\lambda'_i)^2$ where $\lambda'=(\lambda'_1\ge\ldots\ge\lambda'_i\ge\ldots)$ is the conjugate partition and $\vert\lambda\vert\coloneqq\sum_i\lambda_i$. 

Let $\mathcal{E}nd_{r,d}^{nilp,+}$ denote the stack of HN-nonnegative bundles on $X$ of rank $r$ and degree $d$ with nilpotent endomorphisms.

In view of the discussion in Section \ref{section330}, we obtain a version of \cite[Theorem~1.4.1]{FSS1} valid in arbitrary characteristic.
\begin{Th}\label{nilp}
We have the following identity in $\overline{\Mot}(k)[[z,w]].$
\[\sum_{r,d\ge 0}[\mathcal{E}nd^{nilp,+}_{r,d}]w^rz^d=\Omega_X^{Sch,mot}.\]
\end{Th}
When the field $k$ is finite, the above theorem is the motivic version of one of the main intermediate results of \cite{Sch} (see \cite[Proposition 2.2 and (5.20)]{Sch}).
\subsection{Explicit formulas for the motivic classes of the stacks of Higgs bundles}\label{section45}
Finally we obtain the same formula for the motivic classes of the stacks of semistable Higgs bundles as in \cite[Theorem 1.3.3]{FSS1}.  We remove the restriction on the characteristic of the field.

Recall from Theorem \ref{prestructure} that $\overline{\Mot}(k)$ is a pre-$\lambda$-ring. Thus by Section \ref{section23}, we have the plethystic operations $\Exp$ and $\Log$.

We introduce the elements $B_{r,d}\in\overline{\Mot}(k)$ by the formula
\begin{equation}\label{Brd}
\sum_{r,d\in\mathbb{Z}_{\ge 0}\atop(r,d)\neq(0,0)}B_{r,d}w^rz^d=\mathbb{L}\Log\Omega_X^{Sch,mot}\in\overline{\Mot}(k)[[z,w]].
\end{equation}
Now let $\tau\ge 0$ be a rational number, we define $H_{r,d}\in\overline{\Mot}(k)$ by
\begin{equation}\label{Hrd}
\sum_{d/r=\tau}\mathbb{L}^{(1-g)r^2}H_{r,d}w^rz^d=\Exp\left(\sum_{d/r=\tau}B_{r,d}w^rz^d\right).
\end{equation}
Notice that we use the fact that if all monomials in a series have slope $\tau$, then all monomials in the plethystic exponent of this series also have slope $\tau$.

Recall that we have fixed a smooth projective geometrically connected curve $X$ of genus $g$. We have a result similar to \cite[Theorem 1.3.3]{FSS1} for the motivic classes of the moduli stacks of semistable Higgs bundles.
\begin{Th}\label{higgsformula}
Let $k$ be a field of arbitrary characteristic and let $X$ be a smooth projective geometrically connected curve over $k$. Assume that there exists a $k$-rational divisor of degree $1$ on $X$. We have
\begin{enumerate}[(i)]
\item For $r>0$ and $d>r(r-1)(g-1)$, the element $H_{r,d}$ are periodic in $d$ with period $r$, i.e., $H_{r,d}=H_{r,d+r}$.
\item For any $r>0$ and any $d$, if $e$ is large enough $($suffices to take $e>(r-1)(g-~1)-d/r)$, then we have
\[[\mathcal{H}iggs^{ss}_{r,d}]=H_{r,d+er}.\]
\end{enumerate}
\end{Th}
\begin{proof}
The proof of \cite[Section 6.4]{FSS1} goes through in view of our discussion in Section \ref{section330}. Repeating almost literally the arguments of \cite[Sections 3.5-3.8,~6.4]{FSS1}, applying the homomorphism $\overline{\Mot}^{naive}(\mathcal{X})\to\overline{\Mot}(\mathcal{X})$ and using our theory of ``corrected'' motivic classes, we derive the theorem from Theorem \ref{nilp}.
\end{proof}
\begin{Rem}\label{MS0}
Recall that $g$ is the genus of $X$. The volumes of the stacks of Higgs bundles $\vert\mathcal{H}iggs^{ss}_{r,d}(\mathbb{F}_q)\vert$ viewed as functions of $q$ and the Weil numbers of $X$ belong to a certain $\lambda$-ring $K_g$ (see Section \ref{section5} below for more details). Thus we can, following Mozgovoy and Schiffmann, define the Donaldson-Thomas invariants $\Omega(r,d)$ by the following formula
\[\sum_{d/r=\tau}\frac{\Omega(r,d)}{q-1}w^rz^d=\Log\left(\sum_{d/r=\tau}q^{(1-g)r^2}\vert\mathcal{H}iggs^{ss}_{r,d}(\mathbb{F}_q)\vert w^rz^d\right).\]
On the other hand, define the invariants $\Omega^{+}_{r,d}$ by
\begin{equation}
\sum_{r,d\in\mathbb{Z}_{\ge 0}\atop(r,d)\neq(0,0)}\Omega^+_{r,d}w^rz^d=(q-1)\Log\left(\sum_{\lambda}q^{(g-1)\langle\lambda,\lambda\rangle}J_{\lambda}(z)H_{\lambda}(z)w^{\vert\lambda\vert}\right),   
\end{equation}
where $J_{\lambda}$ and $H_{\lambda}$ are defined in \cite[Section 1.4]{Sch}.

Then, according to \cite[Theorem 1.1(ii)]{MS1}, we have \[\Omega(r,d)=q\Omega_{r,d}^+.\] whenever $d$ is large enough. This formula gives an explicit answer for the volumes $\vert\mathcal{H}iggs^{ss}_{r,d}(\mathbb{F}_q)\vert$. It is not hard to derive this formula from our Theorems \ref{countingstack} and \ref{higgsformula}. We suppress the details because in the next section we will give a simpler formula for $\vert\mathcal{H}iggs^{ss}_{r,d}(\mathbb{F}_q)\vert$ based on results of Mellit.
\end{Rem}
\section{Mellit's simplification}\label{section5}
We would like to use Mellit's results \cite{Mel1,Mel2} to obtain a formula for the motivic classes of the stacks of Higgs bundles in the universal $\lambda$-ring quotient that are much simpler than the formulas in Section \ref{section45}. We also give a simpler formula for $\vert\mathcal{H}iggs^{ss}_{r,d}(\mathbb{F}_q)\vert$.
\subsection{Motivic classes of the stacks of Higgs bundles in the universal $\lambda$-ring quotient}
Recall from Theorem \ref{prestructure} that we have a pre-$\lambda$-ring structure on the ring $\overline{\Mot}(k)$ such that for a reduced quasi-projective scheme $X$ we have
\[\lambda_X(z)\coloneqq(\zeta_X(-z))^{-1}.\]
Recall the definition of $\overline{\Mot}^{\lambda-ring}(k)$, see Section \ref{section260}. 
The following discussion is a no-puncture version of \cite[Section 8.6-8.7]{FSS3}. Recall the definitions of the motivic $L$-function, $\Omega_X^{mot}$ and $\Omega_X^{Sch,mot}$, see (\ref{Lfunction}), (\ref{omegamot}) and (\ref{schiff}). Set 
\[\mathbb{H}^{mot}_X\coloneqq(1-z^2)\Log\Omega^{mot}_X,\]
and let $\mathbb{H}^{
Sch,mot}_X$ be the image of $(1-z)\Log\Omega^{Sch,mot}_X$ under the projection $\overline{\Mot}(k)\to\overline{\Mot}^{\lambda-ring}(k)$. 

The following proposition is a no-puncture version of \cite[Corollary 8.7.4]{FSS3}.
\begin{Prop}[Fedorov--Soibelman--Soibelman]\label{twoh}
The coefficients of $\mathbb{H}^{
Sch,mot}_X$ and $\mathbb{H}^{mot}_X$ at $w^r$ belong to the ring  $\overline{\Mot}^{\lambda-ring}(k)[z^{\pm1}]$ and we have
\[\mathbb{H}^{
Sch,mot}_X\Big\vert_{z=1}=\mathbb{H}^{mot}_X\Big\vert_{z=1}.
    \]
\end{Prop}
\begin{proof}
This is a particular case of \cite[Corollary 8.7.4]{FSS3} when $\delta=0$ and $D=\emptyset$.
\end{proof}
In characteristic zero, the following theorem is a particular case of \cite[Theorem~1.13.1]{FSS3} when $\delta=0,D=\emptyset$ and $\varepsilon=0$. Using our definition of the ring of motivic classes and the formulas in Theorem \ref{higgsformula}, we will see that the theorem is valid in arbitrary characteristic. It gives a formula for the motivic classes of
moduli stacks of semistable Higgs bundles in $\overline{\Mot}^{\lambda-ring}(k)$.
\begin{Th}\label{degenerate1}
Let $k$ be a field of arbitrary characteristic and let $X$ be a smooth projective geometrically connected curve over $k$. Assume that there exists a $k$-rational divisor of degree $1$ on $X$. Let $d/r=d_0/r_0$ where $r_0$ and $d_0$ are coprime integers with $r_0>0$. The image of the motivic class $[\mathcal{H}iggs_{r,d}^{ss}]$ in $\overline{\Mot}^{\lambda-ring}(k)$ is equal to the coefficient at $w^r$ in
\[
\mathbb{L}^{(g-1)r^2}\Exp\left(\mathbb{L}\left(\mathbb{H}^{mot}_X\Big\vert_{z=1}\right)_{[r_0]}\right),
\]
where $A(w)_{[r_0]}$ stands for the sum of the monomials whose exponents are multiples of~$r_0$.
\end{Th}
\begin{proof}
According to Theorem \ref{higgsformula}, for $e$ large enough, the motivic class of $\mathcal{H}iggs_{r,d}^{ss}$ in $\overline{\Mot}(k)$ is equal to the coefficient at $w^rz^{d+er}$ in
\[
\mathbb{L}^{(g-1)r^2}\Exp\left(\mathbb{L}\Log\Omega^{Sch,mot}_X\right).
\]
Applying the homomorphism $\overline{\Mot}(k)\to\overline{\Mot}^{\lambda-ring}(k)$, by definition of $\mathbb{H}^{Sch,mot}_X$ we
see that for $e$ large enough, the motivic class of $\mathcal{H}iggs_{r,d}^{ss}$ in $\overline{\Mot}^{\lambda-ring}(k)$ is equal to the coefficient at $w^rz^{d+er}$ in
\[
\mathbb{L}^{(g-1)r^2}\Exp\left(\mathbb{L}\frac{\mathbb{H}^{Sch,mot}_X}{1-z}\right).
\]
The rest of the proof is completely similar to that of \cite[Lemma 9.2.3]{FSS3} in view of Proposition \ref{twoh}.
\end{proof}
\subsection{Counting Higgs bundles over finite fields}
Recall that $S_n$ is the permutation group of $n$
elements. Let $g\ge1$ be an integer. Recall some constructions from \cite[Section 8.3]{FSS3}. We consider the ring of Laurent polynomials in $g+1$ variables
$\mathbb{Z}[Q^{\pm1},\alpha_1^{\pm1},\ldots,\alpha_g^{\pm1}]$ ($Q$ was denoted by $q$ in \cite[Section 8.3]{FSS3}). Now we describe the action of the semi-direct product $W_g\coloneqq S_g\rtimes(\mathbb{Z}/2\mathbb{Z})^g$ on this ring as follows: $W_g$
fixes $Q$, $S_g$ permutes $\alpha_i$, while the $i$-th copy of $\mathbb{Z}/2\mathbb{Z}$ maps $\alpha_i$ to $Q\alpha_i^{-1}$
and $\alpha_j$ to itself whenever $j\neq i$. We denote by $R_g$ the ring of invariants:
\[R_g\coloneqq\mathbb{Z}[Q^{\pm1},\alpha_1^{\pm1},\ldots,\alpha_g^{\pm1}]^{W_g
}.\]
For convenience, we set $\alpha_{i+g}\coloneqq Q\alpha_i^{-1}$ for $i=1,\ldots,g$. We define a \emph{universal $L$-function}
\[L^{univ}_{g}(z)\coloneqq\prod_{i=1}^{2g}(1-\alpha_iz)=\prod_{i=1}^{g}(1-\alpha_iz)\prod_{i=1}^{g}(1-Q\alpha_i^{-1}z)\in R_g[z].\]
We denote by $K_g$ the localization of $R_g$ with respect to the multiplicative set generated
by $Q^i-1$, where $i\ge1$.

For a curve $X$ over $\mathbb{F}_q$, by \cite[Proposition 8.3.1(i)]{FSS3}, similarly to the proof of \cite[Proposition 8.5.1(i)]{FSS3}, we define the evaluation morphism $ev_X\colon K_g\to\mathbb{Q}$ sending $Q$ to $q\in\mathbb{Q}$, and the coefficient of $L^{univ}_{g}$ at $z^i$ to the coefficient of $L_X$ at $z^i$ for $i=1,\ldots,2g$, where
\[L_X(z)\coloneqq\left(\sum_{n\ge 0}\vert Sym^nX(\mathbb{F}_q)\vert\cdot z^n\right)(1-z)(1-qz)\]
is the ``usual'' $L$-function on $X$.

Set (cf. Section \ref{section14})
\[\Omega_{g}^{univ}\coloneqq\sum_{\lambda}w^{\vert\lambda\vert}Q^{g\langle\lambda,\lambda\rangle}\prod_{\square\in\lambda}\frac{L^{univ}_{g}(z^{2a(\square)+1}Q^{-l(\square)-1})}{(z^{2a(\square)+2}-Q^{l(\square)})(z^{2a(\square)}-Q^{l(\square)+1})}\in K_g[[z,w]].\]
We define inverse bijections
\[\Exp_g\colon K_g[[z,w]]^{+}\to(1+K_g[[z,w]]^{+})^{\times}\]
\[\Log_g\colon (1+K_g[[z,w]]^{+})^{\times}\to K_g[[z,w]]^{+}\]
similarly to the definitions of $\Exp$ and $\Log$ in Section~\ref{realsection3}.

Set 
\[\mathbb{H}^{univ}_{g}\coloneqq(1-z^2)\Log_g\Omega^{univ}_{g}.\]
Note that we can evaluate $\mathbb{H}^{univ}_{g}$ at $z=1$, because the coefficients of $\mathbb{H}^{univ}_{g}$ at $w^r$ belong to $K_g[z^{\pm1}]$, see \cite[Theorem 8.7.1(i)]{FSS3}.
\begin{Th}\label{finitefieldhiggs}
Let $X$ be a smooth projective geometrically connected curve over $\mathbb{F}_q$. Assume that there exists an $\mathbb{F}_q$-rational divisor of degree $1$ on $X$. Let $d/r=d_0/r_0$ where $r_0$ and $d_0$ are coprime integers with $r_0>0$. The volume $\vert\mathcal{H}iggs^{ss}_{r,d}(\mathbb{F}_q)\vert$ is the image under $ev_X$ of the coefficient at $w^r$ in
\[
Q^{(g-1)r^2}\Exp_g\left(Q\left(\mathbb{H}^{univ}_{g}\Big\vert_{z=1}\right)_{[r_0]}\right),
\]
where $A(w)_{[r_0]}$ stands for the sum of the monomials whose exponents are multiples of~$r_0$.
\end{Th}
Before giving the proof, we need some preliminaries. We define $V_g$ as the localization of $R_g[z]$ with respect to the multiplicative set generated by $Q^i-z^j$,
where $i,j\ge0$, $(i,j)\neq(0,0)$. Note that $K_g[z]\subset V_g$. 

Let $\mu\colon V_g\to\overline{\Mot}^{\lambda-ring}(\mathbb{F}_q)[[z]]$ be the homomorphism of $\lambda$-rings as in \cite[Proposition 8.5.1(iii)]{FSS3}, sending $L^{univ}_g$ to $L^{mot}_X$ (see \cite[Proposition 8.5.1(i)]{FSS3}).

We denote by $\overline{\Mot}_{fin-type}^{\lambda-ring}(\mathbb{F}_q)$ the image of the composition \[\Mot(\mathbb{F}_q)\to\overline{\Mot}(\mathbb{F}_q)\to\overline{\Mot}^{\lambda-ring}(\mathbb{F}_q).\] By Theorem \ref{countingstack}, the counting measure $\#$ factors through $\overline{\Mot}_{fin-type}^{\lambda-ring}(\mathbb{F}_q)$.

We abuse notation and let 
\[\#\colon \overline{\Mot}_{fin-type}^{\lambda-ring}(\mathbb{F}_q)[[z]]\to\mathbb{Q}[[z]]\]
be the unique extension of $\#$ sending $z$ to $z$.
Similarly we abuse notation and let 
\[ev_X\colon V_g\to\mathbb{Q}[[z]]\]
be the unique extension of the evalutation morphism $ev_X\colon K_g\to\mathbb{Q}$ sending $z$ to $z$.
\begin{Lemma}\label{finitecount}
We have $\#\circ\mu=ev_X$.
\end{Lemma}
\begin{proof}
Since $V_g$ is a localization of $R_g[z]$, it suffices to check that we have a commutative diagram:
\[\begin{tikzcd}
R_g\arrow{dr}{ev_X}\arrow{r}{\mu}& \overline{\Mot}_{fin-type}^{\lambda-ring}(\mathbb{F}_q)\arrow{d}{\#}\\
  & \mathbb{Q}.
\end{tikzcd}\]
By \cite[Proposition 8.3.1(i)]{FSS3}, $R_g$ is generated by the coefficients of $L^{univ}_{g}$ as a $\mathbb{Z}[Q^{\pm1}]$-algebra, so it suffices to check the commutativity of the above diagram for $L^{univ}_{g}$ and $Q$. For $Q$ we have
\[\#\circ\mu(Q)=\#(\mathbb{L})=q=ev_X(Q).\]
Now we have the following sequence of equalities
\[\#\circ\mu(L^{univ}_{g}(z))=\#(L^{mot}_X(z))=L_X(z)=ev_X(L^{univ}_{g}(z)),\]which follow from the definitions of $\mu$, of the counting measure $\#$, and of $ev_X$ respectively.
\end{proof}
\begin{proof}[Proof of Theorem \ref{finitefieldhiggs}]
Applying $\mu$ coefficient-wise to $\Omega_{g}^{univ}$, we get $\Omega_X^{mot}$, see \cite[Section 8.6]{FSS3}. Since $\mu$ is a homomorphism of pre-$\lambda$-rings, we have
\[\mu\left(\mathbb{H}^{univ}_{g}\right)=\mathbb{H}^{mot}_X.\]
We denote by $H^{univ}_r$ the coefficient at $w^r$ in
\[Q^{(g-1)r^2}\Exp_g\left(Q\left(\mathbb{H}^{univ}_{g}\Big\vert_{z=1}\right)_{[r_0]}\right).
\]
By Theorem \ref{degenerate1}, we have 
\[\mu(H^{univ}_r)=[\mathcal{H}iggs_{r,d}^{ss}]\in\overline{\Mot}^{\lambda-ring}(\mathbb{F}_q).\]
Now, applying $\#$ to the above equation and using Lemma \ref{finitecount}, we obtain the result.
\end{proof}
\begin{Rem}
This result is very similar to the no-puncture case of \cite[Corollary~7.8]{Mel3}. However, we obtained this result from the motivic statement, namely, from Theorem \ref{degenerate1}.  
\end{Rem}

\Addresses
\end{document}